\documentclass[a4paper,12pt]{article}
\usepackage[T1]{fontenc}
\usepackage{amsmath, graphicx, color, amsthm, tikz}
\usepackage[font=small,labelfont=bf,width=12.5cm]{caption}
\usepackage{algorithmic, algorithm}

\newtheorem{theorem}{Theorem}[section]
\newtheorem{lemma}[theorem]{Lemma}
\newtheorem{corollary}[theorem]{Corollary}
\newtheorem{proposition}[theorem]{Proposition}
\newtheorem{conjecture}[theorem]{Conjecture}
\newtheorem{question}[theorem]{Question}

\numberwithin{equation}{section}

\title{Odd edge-colorings of subdivisions of odd graphs}

\author
{
	Mirko Petru\v{s}evski \thanks{Department of Mathematics and Informatics, Faculty of Mechanical Engineering - Skopje, Republic of Macedonia. E-Mail: \texttt{mirko.petrushevski@gmail.com}},
	\quad
	Riste \v{S}krekovski\thanks{FMF, University of Ljubljana \& Faculty of Information Studies, Novo mesto, Slovenia. E-Mail: \texttt{skrekovski@gmail.com}}
}

\begin{document}
\maketitle

\begin{abstract}
An odd graph is a finite graph all of whose vertices have odd degrees. A graph $G$ is decomposable into $k$ odd subgraphs if its edge set can be partitioned into $k$ subsets each of which induces an odd subgraph of $G$. The minimum value of $k$ for which such a decomposition of $G$ exists is the odd chromatic index, $\chi_{o}'(G)$, introduced by Pyber (1991). For every $k\geq\chi_{o}'(G)$, the graph $G$ is said to be odd $k$-edge-colorable. Apart from two particular exceptions, which are respectively odd $5$- and odd $6$-edge-colorable, the rest of connected loopless graphs are odd $4$-edge-colorable, and moreover one of the color classes can be reduced to size $\leq2$. In addition, it has been conjectured that an odd $4$-edge-coloring with a color class of size at most $1$ is always achievable. Atanasov et al. (2016) characterized the class of loopless subcubic graphs in terms of the value $\chi_{o}'(G)\leq4$. In this paper, we extend their result to a characterization of all loopless subdivisions of odd graphs in terms of the value of the odd chromatic index. This larger class $\mathcal{S}$ is of a particular interest as it collects all `least instances' of non-odd graphs.  As a prelude to our main result, we show that every connected graph $G\in \mathcal{S}$ requiring the maximum number of four colors, becomes odd $3$-edge-colorable after removing a certain edge. Thus, we provide support for the mentioned conjecture by proving it for all subdivisions of odd graphs. The paper concludes with few problems for possible further work.
\end{abstract}

\medskip



\section{Introduction}

\subsection{Basic terminology}

All considered graphs $G=(V(G),E(G))$ are undirected and finite, loops and parallel edges are allowed. We follow~\cite{BonMur08} for any terminology and notation not defined here. The parameters $n(G)=|V(G)|$ and $m(G)=|E(G)|$ are called the \textit{order} and the \textit{size} of $G$, respectively. A graph of order $1$ is \textit{trivial}, and a graph of size $0$ is \textit{empty}. A path or cycle is either \textit{odd} or \textit{even} depending on the parity of its size. A path (resp. an edge) with endvertices $x$ and $y$ is referred to as an \textit{$x$-$y$ path} (resp. an \textit{$x$-$y$ edge}). Given a path $P$ and vertices $x,y\in V(P)$, the $x$-$y$ subpath of $P$ is denoted $xPy$.
For every vertex $v\in V(G)$, $E_G(v)$ denotes the set of edges incident with $v$, and the size of $E_G(v)$ (every loop being counted twice) is the \textit{degree}, $d_{G}(v)$, of $v$ in $G$. The maximum and minimum vertex degree in $G$ are denoted by $\Delta(G)$ and $\delta(G)$, respectively.  A graph $G$ is \textit{subcubic} if $\Delta(G)\leq 3$. Each vertex $v$ of even (resp. odd) degree $d_{G}(v)$ is an \textit{even} (resp. \textit{odd}) vertex. In particular, if $d_G(v)$ equals $0$ (resp. $1$), we say that $v$ is an \textit{isolated} (resp. \textit{pendant}) vertex of $G$. Any vertex of degree $d$ is called a \textit{$d$-vertex}. A graph is \textit{even} (resp. \textit{odd})
whenever all its vertices are even (resp. odd).
The set of neighboring vertices of $v\in V(G)$ is denoted by $N_G(v)$. For every $u\in N_G(v)$, the edge set $E_G(u)\cap E_G(v)$ is the \textit{$u$-$v$ bouquet} in $G$, with notation $\mathcal{B}_{uv}$.  The maximum size of a bouquet in $G$ is its \textit{multiplicity}, $\mu(G)$. A graph $G$ is \textit{simple} if it is loopless and of multiplicity at most $1$.

For $X\subseteq V(G)\cup E(G)$, $G-X$ is the subgraph of $G$ obtained by removing $X$; we abbreviate $G-\{x\}$ to $G-x$. Similarly, given a subgraph $H\subseteq G$, $H+X$ is the subgraph of $G$ obtained by adding to $H$ all the vertices and edges from $X$. A spanning subgraph of $G$ is also called a \textit{factor} of $G$.

To \textit{split} a vertex $v$ is to replace it by two (not necessarily adjacent) vertices $v'$ and $v''$, and to replace each edge incident to $v$ by an edge incident to either $v'$ or $v''$ (but not both, unless the edge is a loop at $v$), the other end of the edge remaining unchanged. A vertex of positive degree can be split in several ways, so the resulting graph is not unique in general. Another local operation on graph $G$ is to \textit{suppress} a $2$-vertex $v$. The modified graph $G\%v$ is obtained from $G-v$ by adding an edge between the neighbors of $v$ (the new edge is a link unless $N_G(v)$ is a singleton).

The \textit{connectivity}, $\kappa(G)$,  of a graph $G$ is the minimum size of a subset $S\subseteq V(G)$  such that $G-S$ is disconnected or of order $1$. A graph is said to be \textit{$k$-connected} if its connectivity is at least $k$.
A vertex $v\in V(G)$ is a \textit{cutvertex} of $G$ if $G-v$ has more (connected) components than $G$. If $V_1,\ldots,V_k$ are the vertex sets of all components of $G-v$, then for $i=1,\ldots,k$, the induced subgraph $G[V_i\cup \{v\}]$ is called a \textit{$v$-lobe} of $G$.
 A \textit{block graph} is a connected graph without any cutvertices. Given a nontrivial connected graph $G$, a maximal block subgraph is a \textit{block} of $G$. Thus each block is either $2$-connected or a bouquet, and each cycle is entirely within a single block. For a block $B$ of $G$, each vertex $v\in V(B)$ which is not a cutvertex of $G$ is an \textit{internal vertex} of $B$ (and of $G$). The collection of internal vertices of $B$ is denoted by $\mathrm{Int}_G(B)$.
 If $V(B)$ contains at most one cutvertex of $G$ then $B$ is an \textit{end-block}.
 Any connected graph $G$ is associated with a bipartite graph $B(G)$ having bipartition
$(\mathcal{B},\mathcal{V})$, where $\mathcal{B}$ is the set of blocks of $G$ and $\mathcal{V}$ the set of cutvertices of $G$, a
block $B\in \mathcal{B}$ and a cutvertex $v\in \mathcal{V}$ being adjacent in $B(G)$ if and only if $B$ contains
$v$. The graph $B(G)$ is connected and acyclic, the former because $G$ is connected and the latter because a cycle in $B(G)$ would correspond to
a cycle in $G$ passing through two or more blocks.
The graph $B(G)$ is therefore a tree, called the \textit{block-tree} of $G$.
If $\mathcal{V}\neq\emptyset$, the end-blocks of $G$ correspond to the leaves of its block-tree. Every vertex $v$ of a block graph $G$ has a neighbor among the internal vertices of each end-block of $G-v$.

For a nonempty subset $X\subset V(G)$, the \textit{edge cut} $\partial(X)$ is the set of edges with one endvertex in $X$ and the other endvertex in $V(G)\backslash X$; in case $X$ is a singleton, we speak of a \textit{trivial} edge cut $\partial(X)$. The \textit{edge-connectivity}, $\kappa'(G)$, of a nontrivial graph $G$ is the minimum size of a subset $S\subseteq E(G)$ such that $G-S$ is disconnected; equivalently, $\kappa'(G)$ is the minimum size of an edge cut in $G$.  A \textit{$k$-edge cut} is an edge cut of size $k$; a $1$-edge cut is also called a \textit{bridge}. If $vw$ is a bridge and the vertex $w$ is not the only neighbor of the vertex $v$, then $v$ is a cutvertex of $G$. A graph is said to be \textit{$k$-edge-connected} if its edge-connectivity is at least $k$.
A $k$-edge-connected
graph is termed \textit{essentially $(k+1)$-edge-connected} if all of its $k$-edge cuts
are trivial.


\subsection{Odd edge-colorings and odd chromatic index}

An assignment $\varphi: E(G)\to S$ is an \textit{edge-coloring of $G$} with \textit{color set} $S$. If $|S|\leq k$, we speak of a $k$-edge-coloring $\varphi$. The nature of the colors is irrelevant, and it is conventional  to use $S=[k]:=\{1,2,\ldots,k\}$ for a color set of size $k$. For each color $c\in S$, $E_c(G,\varphi)$ denotes the \textit{color class of $c$}, that is, the set $\varphi^{-1}(c)$ of edges colored by $c$. Whenever $G$ and $\varphi$ are clear from the context, we denote the color class of $c$ simply by $E_c$.  Given an edge-coloring $\varphi$ and a vertex $v$ of $G$, we say that a color $c$ \textit{appears at $v$} if $E_c\cap E_G(v)\neq \emptyset$. Any decomposition $\{H_1,\ldots, H_k\}$ of $G$ can alterably be interpreted as a $k$-edge-coloring of $G$ for which the color classes are $E(H_1),\ldots, E(H_k)$.

An \textit{odd edge-coloring} of a graph $G$ is an edge-coloring such that each nonempty color class $E_c$ induces an odd subgraph of $G$. In other words, at each vertex $v$, for any appearing color $c$ the degree $d_{G[E_c]}(v)$ is odd. Equivalently, an odd edge-coloring can be seen as a decomposition of $G$ into
(edge-disjoint) odd subgraphs. As usual, we are most interested in the least number of colors necessary to create such a coloring.
An odd edge-coloring of $G$ using at most $k$ colors is referred to as an \textit{odd $k$-edge-coloring}, and if such a coloring exists we say that $G$ is
\textit{odd $k$-edge-colorable}. Whenever $G$ admits an odd edge-coloring, the \textit{odd chromatic index}, $\chi_{o}'(G)$, is
defined to be the minimum integer $k$ for which $G$ is odd $k$-edge-colorable.

It is obvious that a necessary and sufficient condition for odd edge-colorability of $G$ is the absence of vertices incident only to loops. Apart from this, the presence of loops
does not influence the existence nor changes the value of the index $\chi_{o}'(G)$. Therefore, the class of loopless graphs comprises a natural framework for the study of the odd chromatic index.

\begin{figure}[ht!]
	$$
		\includegraphics{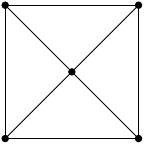}
	$$
	\caption{The wheel $W_4$ is a simple graph with $\chi'_o(W_4)=4$.}
	\label{fig:pyb}
\end{figure}

As a notion, odd edge-coloring was introduced by Pyber in his survey on graph coverings~\cite{Pyb91}.
The mentioned work concerns simple graphs and (among other results) contains a proof of the following.

\begin{theorem}[Pyber, 1991]
	\label{thm:pyb}
	For every simple graph $G$, it holds that
	$\chi_o'(G) \le 4$.
\end{theorem}

Pyber observed that the established upper bound is realized by the wheel on four spokes $W_4$ (see Figure~\ref{fig:pyb}). However, this upper bound of four colors does not apply to the class of all looplees graphs $G$. For instance, Figure~\ref{fig:shan} depicts four graphs with the following characteristic property: each of their odd subgraphs is of order $2$ and size $1$, that is, a copy of $K_2$. Consequently, for each of these graphs the odd chromatic index equals the size.

\begin{figure}[ht!]
	$$
		\includegraphics{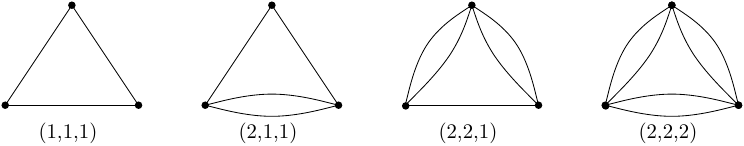}
	$$
	\caption{Four Shannon triangles (the smallest one of each type).}
	\label{fig:shan}
\end{figure}

As defined in~\cite{LuzPetSkr15}, a \textit{Shannon triangle} is a loopless graph on three pairwise adjacent vertices. And if $p,q,r$ are parities of the sizes of its bouquets in non-increasing order, with $2$ (resp. $1$) denoting an even-sized (resp. odd-sized) bouquet, then $G$ is a Shannon triangle of \textit{type $(p,q,r)$}. Figure~\ref{fig:shan} depicts (from left to right) the smallest, in terms of size, Shannon triangle of type $(1,1,1)$, $(2,1,1)$, $(2,2,1)$, and $(2,2,2)$, respectively. It is straightforward that if $G$ is a Shannon triangle of type $(p,q,r)$, then
\begin{equation}
    \label{eqn:shan}
\chi'_{o}(G)=p+q+r\,.
\end{equation}
The main result of~\cite{LuzPetSkr15} tells that six colors suffice for an odd edge-coloring of any loopless graph. Furthermore, it characterizes when six colors are necessary.
\begin{theorem}
    \label{char:6}
    For every connected loopless graph $G$, it holds that
	$\chi_{o}'(G) \leq 6$.
	Moreover, equality is attained if and only if $G$ is a Shannon triangle of type $(2,2,2)$.
\end{theorem}

Recently, the following improvement of Theorems~\ref{thm:pyb} and~\ref{char:6} has been shown in~\cite{Pet18}.

\begin{theorem}
 \label{odd 4-edge-colorability}
 Let $G$ be a connected loopless graph that is not a Shannon triangle of type $(2,2,1)$ or $(2,2,2)$. Then $G$ admits an odd edge-coloring with color set $\{1,2,3,4\}$ such that the color class $E_4$ satisfies two additional conditions:
\begin{enumerate}
\item[$(i)$] $|E_4|\in\{0,1,2\}$, and if $|E_4|=2$ then the pair of edges colored by $4$ are at distance $2$ (i.e., are second-neighbors in the line graph);
\item[$(ii)$] if $\mathcal{B}_{xy}\cap E_4\neq \emptyset$ then another common color (besides $4$) appears at $x$ and $y$.
\end{enumerate}

\end{theorem}

\begin{figure}[ht!]
	$$
		\includegraphics{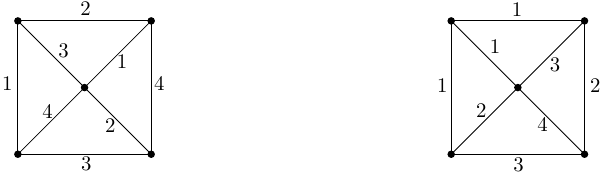}
	$$
	\caption{Two odd edge-colorings of $W_4$ that satisfy conditions $(i)$ and $(ii)$ from Theorem~\ref{odd 4-edge-colorability}. In the coloring depicted on the left, $|E_4|=2$ and  at both endvertices of any edge colored by $4$ either the color $2$ or the color $3$ appears. In the coloring depicted on the right, $|E_4|=1$ and  at both endvertices of the only edge colored by $4$ each of the colors $2$ and $3$ occurs.}
	\label{fig:wheel}
\end{figure}

It is not known whether there exists a connected graph $G$ with $\chi'_o(G)=4$ that does not admit an odd $4$-edge-coloring with a color class of size $1$. In this regard, the following has been conjectured in~\cite{PetSkr21}.

\begin{conjecture}
    \label{conj:finiteness}
Every connected graph $G$ with $\chi'_o(G)=4$ becomes odd $3$-edge-colorable by removing a particular edge.
\end{conjecture}

Regarding odd $2$-edge-colorability of graphs, Kano et al.~\cite{KanKatVar18} have shown the following.

\begin{theorem}
    \label{Kano}
The decision problem whether a given graph $G$ is odd $2$-edge-colorable is solvable in polynomial time. Moreover, in the affirmative case, such a coloring can be found in polynomial time.
\end{theorem}

In view of Theorem~\ref{odd 4-edge-colorability}, the decision problem whether a given graph $G$ is odd $4$-edge-colorable is also solvable in polynomial (in fact, linear) time. Moreover, its proof can be used as an efficient algorithm for exhibiting such a coloring. The analogous complexity questions regarding odd $3$-edge-colorability of general graphs are still open. Nevertheless, these questions have been answered (in the affirmative) for subcubic graphs. Namely, a complete characterization of the class of loopless subcubic graphs in terms of their odd chromatic index was obtained in~\cite{AtaPetSkr16} through the following:

\begin{theorem}
    \label{thm:characterizationsubcubic}
Let $G$ be a connected loopless subcubic graph. Then
\begin{equation*}
\chi'_o(G)=
\begin{cases}
0 & \text{\quad if\, } G \text{ is empty}\,;\\
1 & \text{\quad if \,} G \text{ is odd} \,;\\
2 & \text{\quad if \,} G \text{ has 2-vertices, with an even number of them on each cycle}\,;\\
4 & \text{\quad if \,} G \text{ is obtainable from a cubic bipartite graph by a single edge subdivision}\,;\\
3 & \text{\quad otherwise}\,.
\end{cases}
\end{equation*}
\end{theorem}

In this paper we focus on loopless subdivisions of odd graphs, which are in some sense the least non-odd graphs among all. Let us denote this collection by $\mathcal{S}$, and similarly, let $\mathcal{O}$ be the class of loopless odd graphs, with the understanding that $\mathcal{O}\subset\mathcal{S}$. Our main result,  Theorem~\ref{thm} at the very end of Section~$4$, is a characterization of the members of $\mathcal{S}$ in terms of the value of their odd chromatic index. Thus, we achieve a generalization of Theorem~\ref{thm:characterizationsubcubic}, and at the same time answer a question raised by the end of \cite{Pet18}. Our findings here also provide support for  Conjecture~\ref{conj:finiteness} over the class $\mathcal{S}$.

\medskip

The rest of the article is divided into four sections. In the next, preliminary one, we collect several `easy' results (most of them previously known). Sections~$3$ and $4$ are devoted to a derivation of our main result - a characterization of $\mathcal{S}$ in terms of the value of the odd chromatic index. The final section briefly conveys some possible directions for further related study.


\section{Preliminaries}

The \textit{edge-complement}, $\widehat{H}$, of a subgraph $H\subseteq G$ is the spanning subgraph $\widehat{H}=G-E(H)$.
A \textit{co-forest} in $G$ is a subgraph whose edge-complement is a forest.
For a graph $G$, let $T$ be an even-sized subset of $V(G)$. Following~\cite{BonMur08}, a spanning subgraph $H$ of $G$ is said to be a \textit{$T$-join} of $G$ if $d_{H}(v)$ is odd for all $v\in T$ and even for all $v\in V(G)\setminus T$. For instance, if $P\subseteq G$ is a nontrivial path with endvertices $x$ and $y$, the spanning subgraph of $G$ with edge set $E(P)$ is an $\{x,y\}$-join of $G$. As another example, every even spanning subgraph is an $\emptyset$-join of $G$. Observe that the symmetric difference of an $S$-join and a $T$-join is an $S\oplus T$-join. (We shall use $\oplus$ to denote both the symmetric difference operation on spanning subgraphs and on sets.) Hence, the symmetric difference, $H\oplus K$, of a $T$-join $H$ and a spanning even subgraph $K$ of $G$ is again a $T$-join. In particular, the removal (resp. addition) of all edges of an edge-disjoint cycle from (resp. to) a $T$-join, produces another $T$-join. Therefore, if a $T$-join of $G$ exists, there also exists such a forest (resp. co-forest). By the handshake lemma, necessary for the existence of a $T$-join is that the intersection of $T$ with the vertex set of every component of $G$ is even-sized, and a straightforward implementation of the  above mentioned facts (see~\cite{Sch03}) is that this condition also suffices. Consequently, given a connected graph $G$ and an even-sized subset $T$ of $V(G)$,
\begin{enumerate}
\item[$(1)$] there exists a $T$-join of $G$ that is a forest;
\item[$(2)$] there exists a $T$-join of $G$ that is a co-forest;
\item[$(3)$] additionally, if $G$ is of even order, then it contains a spanning odd co-forest.
\end{enumerate}

\smallskip

 An edge-coloring $\varphi$ is said to be \textit{odd} (resp. \textit{even}) \textit{at} a vertex $v$ if each color appearing  at $v$ is odd (resp. even). Similarly, we say that $\varphi$ is odd (resp. even) \textit{away from $v$} if $\varphi$ is odd (resp. even) at every vertex $w\in V(G)\backslash\{v\}$, without any assumptions about the behavior of $\varphi$ at $v$ being made.
The following useful result appears in~\cite{Mat06,Pet18}.

\begin{proposition}
    \label{forest}
    Let $v$ be a vertex of a forest $F$. Any local coloring of $E_F(v)$ which uses at most two colors extends to a $2$-edge-coloring of $F$ that is odd away from $v$. In particular, $F$ is odd $2$-edge-colorable.
\end{proposition}

An immediate consequence of Proposition~\ref{forest} is the result below, which concerns a graph all of whose cycles (if any) share a vertex.
 \begin{proposition}
    \label{consequences:forest}
    If $v$ is a vertex of a graph $G$ such that $G-v$ is a forest, then $G$ admits a $2$-edge-coloring that is odd away from $v$. Additionally, if $d_G(v)$ is odd, then $G$ admits an edge-coloring with color set $\{1,2\}$ that is odd away from $v$ and the color $1$ (resp. $2$) is odd (resp. even) at $v$.
 \end{proposition}

 \begin{proof}
 We may assume that $G$ is loopless. It suffices to prove the first part. Split $v$ into $k=d_G(v)$ pendant vertices $v_1,\ldots,v_k$ in order to obtain a forest $F$. By Proposition~\ref{forest}, $F$ admits an odd $2$-edge-coloring. Re-identify $v_1,\ldots,v_k$ into $v$ while keeping the colors on all edges. We thus regain $G$ along with a required edge-coloring.
 \end{proof}

As observed in~\cite{Pyb91}, the odd $2$-edge-colorability of forests implies odd $3$-edge-colorability for all connected graphs of even order, which in turn yields odd $3$-edge-colorability for all graphs with edge-connectivity $1$. The following proof comes from~\cite{Pet18}.
\begin{proposition}
     \label{even order}
If $G$ is a connected graph such that $n(G)$ is even or $\kappa'(G)=1$ then $\chi_{o}'(G)\leq 3$.
\end{proposition}
\begin{proof}
 Let $n(G)$ be even and let $H$ be a spanning odd co-forest of $G$. Take an odd edge-coloring of the forest $\widehat{H}$ with color set $\{1,2\}$ and extend to $E(G)$ by coloring $E(H)$ with $3$. This gives an odd $3$-edge-coloring of $G$.

Assume now that $n(G)$ is odd.
	First we consider the case when the minimum degree $\delta(G)=1$. Select a pendant vertex $u$ and take a spanning odd co-forest $H$ of $G-u$. As $F=G-E(H)$ is a forest, combine an odd $2$-edge-coloring of $F$ with a monochromatic coloring of $E(H)$ that uses a third color.
	
	So suppose that there are no pendant vertices in $G$, but nevertheless $\kappa'(G)=1$. Let $vw$ be a bridge in $G$. Denote by $G_{v}$ and $G_{w}$, respectively,
	the components of $G-vw$ containing $v$ and $w$.
	By the previous case, the subgraphs $G'=G[V(G_{v})\cup\{w\}]$ and $G''=G[V(G_{w})\cup\{v\}]$ admit respective odd $3$-edge-colorings $\varphi'$ and $\varphi''$ with the same color set.
	Moreover, by permuting colors if necessary, we can achieve that $\varphi'(vw)=\varphi''(vw)$.
	Then $\varphi' \cup \varphi''$ is an odd $3$-edge-coloring of $G$.
\end{proof}

The next result may be used to characterize odd $2$-edge-colorability of unicyclic graphs.

\begin{proposition}
Let $G$ be a unicyclic loopless graph, and let $C\subseteq G$ be the (unique) cycle. Then $\chi'_o(G)\leq3$. Moreover, the upper bound is attained if and only if the following two conditions hold simultaneously:
\begin{itemize}
\item[$(i)$] $\{v\in V(C):d_G(v)=2\}$ is odd-sized;
\item[$(ii)$] $\{v\in V(C):d_G(v)\neq2$ and $d_G(v)$ is even$\}=\emptyset$.
\end{itemize}
\begin{proof}
Let us first show that $G$ is odd $3$-edge-colorable. Since $G$ is unicyclic, Proposition~\ref{forest} allows for the assumption that $G$ is connected. Moreover, in view of Proposition~\ref{even order}, we may further assume that $G$ is bridgeless. However, from all assumed it readily follows that $G=C$. Hence $\chi'_o(G)=\chi'(C)\leq3$.

Now we show that fulfilment of the conditions $(i)$ and $(ii)$ is both necessary and sufficient for the equality $\chi'_o(G)=3$ to hold.
If $G=C$ then condition $(ii)$ is clearly met, and the characterization is trivially true (as by then the notions `odd edge-coloring'  and `proper edge-coloring' become equivalent). Assuming $G\neq C$, let $S=\{v\in V(C):d_G(v)\neq2\}$ and $\widehat{S}=V(C)\backslash S$. Denote by $S'$ and $S''$, respectively, the subsets of $S$ comprised of those vertices $v$ for which the degree $d_G(v)$ is odd or even. Observe that for every $v\in S'$, a coloring of $E_C(v)$ extends to an odd $2$-edge-coloring of $E_G(v)$ if and only if it is monochromatic. Otherwise, for every $v\in S''$ each coloring of $E_C(v)$ extends to an odd $2$-edge-coloring of $E_G(v)$. Consequently, in view of Proposition~\ref{forest}, a given $2$-edge-coloring of $C$ extends to an odd $2$-edge-coloring of $G$ if and only if the coloring is dichromatic at each $v\in \widehat{S}$ and monochromatic at each $v\in S'$.

So, if condition $(ii)$ fails to hold, then $\chi'_o(G)\leq2$. Indeed, simply select a vertex $v\in S''$, take a $2$-edge-coloring of $C$ that is monochromatic at each vertex from $S'$ and dichromatic at each vertex from  $V(C)\backslash (S'\cup \{v\})$; by the above observation, such a coloring of $E(C)$ extends to an odd $2$-edge-coloring of $G$.

On the other hand, assuming $(ii)$, odd $2$-edge-colorability of $G$ is equivalent to the existence of a $2$-edge-coloring of $C$ that is dichromatic precisely at each vertex of $\widehat{S}$. The latter is clearly equivalent to the requirement that the set $\{v\in V(C):d_G(v)=2\}$ is even-sized.
\end{proof}
\end{proposition}
\begin{corollary}
    \label{unicycle}
Let $G$ be a connected unicyclic loopless graph, $C\subseteq G$ be the (unique) cycle and let $\{v\in V(C):d_G(v) \text{ is odd}\}=\emptyset$. Then $\chi'_o(G)\leq2$ unless $G=C$ is an odd cycle.
\end{corollary}

 We end the preliminaries with two more already known results (proofs can be found in~\cite{Pet18}).

\begin{proposition}
    \label{co-forest}
In a connected loopless graph $G$, let $v$ be an internal vertex and $e\in E_G(v)$. If $T\subseteq V(G)$ is even-sized, then there exists a $T$-join $H$ of $G$ which is a co-forest such that $E_{\widehat{H}}(v)\subseteq\{e\}$.
\end{proposition}

If additionally the graph $G$ from Proposition~\ref{co-forest} is of even order, we derive the following by setting $T=V(G)$.
\begin{corollary}
    \label{odd co-forest}
   In a connected loopless graph $G$ of even order, let $v$ be an internal vertex and $e\in E_G(v)$. Then there exists a spanning odd co-forest $H$ of $G$ such that $E_{\widehat{H}}(v)\subseteq\{e\}$.

\end{corollary}


\section{Subdivisions of odd graphs}

Recall that $\mathcal{S}$ denotes the class of all loopless subdivisions of odd graphs. The following proposition is an overture to our subsequent study of $\mathcal{S}$ in terms of the value of the odd chromatic index. The final product of the study, our main result, shall be formulated by the end of Section~4. As a warm-up, we commence by showing that four colors always suffice for an odd edge-coloring of any member of $\mathcal{S}$. Moreover, the fourth color can be reduced to at most one appearance per component.
\begin{proposition}
    \label{4-edge-colorability}
Let $v$ be a $2$-vertex of a connected graph $G\in \mathcal{S}$, and let $e\in E_G(v)$. Then $G$ admits an odd edge-coloring with color set $\{1,2,3,4\}$ such that the color class $E_4\subseteq \{e\}$. Moreover, if $\chi'_o(G)=4$ then it holds that:
 \begin{itemize}
 \item [$(i)$] Every $2$-vertex is internal;
 \item [$(ii)$] No $2$-vertices are adjacent.
 \end{itemize}
\end{proposition}
\begin{proof}
If there exists a cutvertex $u$ in $G$ such that $d_G(u)=2$, then $u$ must be incident with two bridges. Consequently, Proposition~\ref{even order} yields odd $3$-edge-colorability of $G$. Assuming $(i)$, the vertex $v$ is internal. Therefore, the graph $G-e$ is connected and of minimum degree $\delta(G-e)=1$, so it admits an odd edge-coloring with color set $\{1,2,3\}$. By assigning the color $4$ to the edge $e$ we obtain the promised coloring of $E(G)$. This proves the first part and, in addition, confirms that $(i)$ is necessary for $\chi'_o(G)=4$.

As for $(ii)$, still assuming $\chi'_o(G)=4$, suppose there is an edge $f$ whose endvertices $u$ and $w$ are $2$-vertices. Let $g_u$ and $g_w$ be the other edges (besides $f$) incident with $u$ and $w$, respectively. Since $g_u$ and $g_w$ are pendant edges in the (connected) graph $G-f$, Proposition~\ref{even order} guarantees that there is an odd edge-coloring $\varphi$ of $G-f$ with color set $\{1,2,3\}$. Extend $\varphi$ to $E(G)$ by assigning $f$ with a color from $\{1,2,3\}\backslash\{\varphi(g_u),\varphi(g_w)\}$. This completes an odd $3$-edge-coloring of $G$, a contradiction.\end{proof}

Note that the first part of Proposition~\ref{4-edge-colorability} supports Conjecture~\ref{conj:finiteness}. For our intended characterization of all members of the class $\mathcal{S}$ in terms of their odd chromatic index, let us denote by $\mathcal{S}_i$ $(i=1,2,3,4)$ the subclass consisting of those $G\in\mathcal{S}$ having $\chi'_o(G)=i$. Clearly, $\mathcal{S}_1$ comprises the class of loopless odd graphs, $\mathcal{O}$. At the other end of the spectrum, the second part of Proposition~\ref{4-edge-colorability} gives a pair of necessary conditions for membership in $\mathcal{S}_4$;  equivalently, it describes two sufficient conditions for odd $3$-edge-colorability of a loopless subdivision of an odd graph. The following result provides another such condition (which shall be useful on more than one occasion later on in this section).

\begin{proposition}
    \label{sufficient condition 3-edge-colorability}
Let $G\in \mathcal{S}$ be a connected graph, and let $C\subseteq G$ be a cycle passing through a $2$-vertex $v$ of $G$. If the cycle $C$ is even or it passes through another $2$-vertex of $G$, then $\chi'_o(G)\leq 3$.
\end{proposition}
\begin{proof}
We argue by contradiction, that is, suppose $G$ is not odd $3$-edge-colorable. Then, by Proposition~\ref{even order}, the order $n(G)$ is odd. So, in view of Proposition~\ref{4-edge-colorability}, we have that the graph $G-v$ is connected and of even order. Let $\mathcal{H}$ be the collection of all spanning odd co-forests of $G-v$. Thus $\mathcal{H}\neq\emptyset$. As $d_G(v)=2$, the neighborhood $N_G(v)$ is either a $1$-set or a $2$-set. We show it is the latter.

\bigskip

\noindent \textbf{Claim 1.} $|N_G(v)|=2$.

\medskip

\noindent Otherwise, the cycle $C$ is of length $2$ (namely, $E(C)=E_G(v)$). Consider a member $H\in\mathcal{H}$, along with its edge-complement in regard to $G$: the former subgraph (the odd co-forest $H$) is odd $1$-edge-colorable, whereas the latter is a unicyclic graph (the unique cycle is $C$) and its component containing the cycle satisfies all assumptions of Corollary~\ref{unicycle}. Hence, $G$ admits an odd edge-coloring that uses at most $1+2=3$ colors, a contradiction.\hfill$\diamond$

\bigskip

Let $N_G(v)=\{u,w\}$. By Proposition~\ref{4-edge-colorability}~$(ii)$, the degrees $d_G(u),d_G(w)$ are odd. Let $P=C-v$ and observe that (by the initial assumptions) $P$ is a $u$-$w$ path in $G-v$ which is even or it passes through a $2$-vertex of $G$. For any $H\in\mathcal{H}$ denote $\widehat{H}=G-v-E(H)$. Since $H$ is a (spanning odd) co-forest of $G-v$, the graph $\widehat{H}$ is a forest. We show next that there is a particular component in $\widehat{H}$.

\bigskip

\noindent \textbf{Claim 2.} \textit{For every $H\in\mathcal{H}$, a component of $\widehat{H}$ is an odd $u$-$w$ path $Q=Q(\widehat{H})$. Moreover, $Q\neq P$.}

\medskip

\noindent  Note that the vertices of odd degree in $\widehat{H}$ are precisely $u,w$ and all $2$-vertices of $G$ that are within $G-v$. Thus, looking at $G-E(H)=\widehat{H}+\{uv,vw\}$, the odd vertices of this graph are precisely the $2$-vertices of $G$ that are $\neq v$; moreover, each such vertex is pendant in regard to $G-E(H)$. There are two possibilities for $G-E(H)$: either it is a forest (if $u,w$ do not share a component of $\widehat{H}$), or it is a unicyclic graph such that no vertex of the cycle has an odd degree. Therefore, in view of Corollary~\ref{unicycle}, the graph $G-E(H)$ is odd $2$-edge-colorable (and consequently $G$ is odd $3$-edge-colorable), unless it is always the case that a component of $\widehat{H}$ is an odd $u$-$w$ path, say $Q$. As we are supposing $\chi'_o(G)>3$, we have thus established the existence of the path-component $Q$.
Let us show that $Q\neq P$. If the cycle $C$ is even, then $P$ and $Q$ are of different parities, and hence cannot be the same. And if $C$ passes through a $2$-vertex of $G$ contained within $G-v$, then so does $P$ but not $Q$ (because every such $2$-vertex of $G$ is a pendant vertex of $\widehat{H}$). \hfill$\diamond$

\bigskip

Let $Q_u$ and $Q_w$ be, respectively, the components of $u$ and $w$ in $P\cap Q$. (Here $P$ and $Q$ are seen as spanning subgraphs of $G-v$ with respective edge sets $E(P)$ and $E(Q)$.) Since $Q\neq P$, the paths $Q_u$ and $Q_w$ are disjoint. The former has $u$, whereas the latter has $w$ as an endvertex. Notice that $Q_u\cup Q_w\subset P$. Say $e_u$ and $e_w$ are, respectively, the first and the last edge lying outside $Q_u\cup Q_w$ on a traversal of $P$ from $u$ to $w$ (it is not excluded that $e_u=e_w$).

 \begin{figure}[ht!]
	$$
		\includegraphics[scale=0.8]{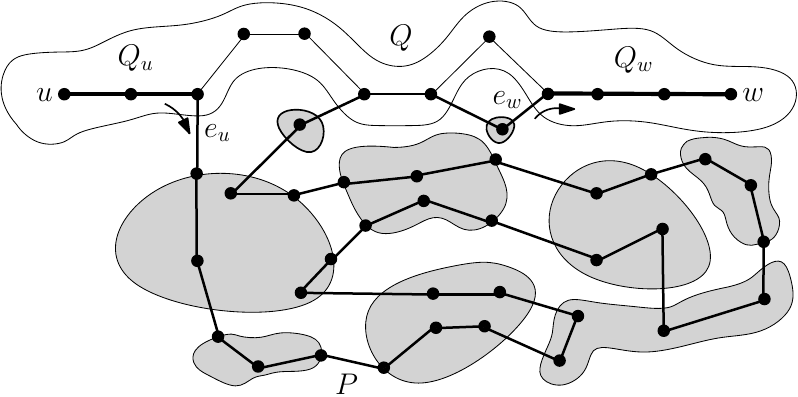}
	$$
	\caption{The path $P$ and the components of the forest $\widehat{H}$. Apart from the path-component $Q$, the rest of the components of $\widehat{H}$ are shaded. The edges of $P$ are depicted as heavier, and those of $Q_u\cup Q_w$ are fat. On a traversal of $P$ from $u$ to $w$, the arrows notify the first embarkment and the last disembarkment of $P-E(Q)$. Since this happens precisely along the edges $e_u$ and $e_w$, respectively, these two cannot be cycle edges of $\widehat{H}\oplus P\oplus Q$. The same is obviously true for any edge of $Q_u\cup Q_w$.}
	\label{fig:PQH}
\end{figure}

Take the symmetric difference $H\oplus P\oplus Q$. The obtained graph is clearly another odd factor of $G-v$, though not necessarily a co-forest. Let $H'$ be a maximal odd factor of $G-v$ subjected to the condition $H\oplus P\oplus Q\subseteq H'$. Then obviously $H'\in\mathcal{H}$. According to Claim~2, an odd $u$-$w$ path $Q'$ constitutes a component of the forest $\widehat{H'}=G-v-E(H')$. Let $Q'_u,Q'_w$ be defined analogously as before, that is, let $Q'_u$ and $Q'_w$ be the respective components of $u$ and $w$ in $P\cap Q'$.

\bigskip

\noindent \textbf{Claim 3.} \textit{$Q_u\subset Q'_u$ and $Q_w\subset Q'_w$.}

\medskip

\noindent Begin by observing that $H'-E(H\oplus P\oplus Q)$ is an even subgraph of $\widehat{H}\oplus P\oplus Q$, the edge-complement of $H\oplus P\oplus Q$ with respect to $G-v$. Also note that $Q_u\cup\{e_u\}\cup \{e_w\}\cup Q_w$ is  fully contained in $\widehat{H}\oplus P\oplus Q$. Moreover, for any vertex $x\in V(Q_u\cup Q_w)$ it holds that $d_{\widehat{H}\oplus P\oplus Q}(x)=d_P(x)\leq2$. In particular, $d_{\widehat{H}\oplus P\oplus Q}(u)=d_{\widehat{H}\oplus P\oplus Q}(w)=1$. Therefore, no edge from $Q_u\cup\{e_u\}\cup \{e_w\}\cup Q_w$ belongs to a cycle contained entirely in $\widehat{H}\oplus P\oplus Q$. Consequently, $H'$ is edge-disjoint from $Q_u\cup\{e_u\}\cup \{e_w\}\cup Q_w$. Equivalently, $Q_u\cup\{e_u\}\cup \{e_w\}\cup Q_w\subseteq P\cap Q'$. It follows that $Q_u\cup\{e_u\}\subseteq Q'_u$ and $Q_w\cup\{e_w\}\subseteq Q'_w$. \hfill$\diamond$

\bigskip

So for any $H\in\mathcal{H}$, there exists another $H'\in\mathcal{H}$ such that $Q_u\cup Q_w\subset Q'_u\cup Q'_w \subset P$.
This is the desired contradiction.
\end{proof}

Our next result concerns odd $2$-edge-colorability of subdivisions of odd graphs, and thus yields a structural characterization of $\mathcal{S}_2$.

\begin{proposition} The following statements are equivalent for every graph $G\in \mathcal{S}$:
\begin{itemize}
\item[$(i)$] $\chi'_o(G)\leq2$;
\item[$(ii)$] For every cycle $C$ of $G$ the set $\{v:v\in V(C) \text{ and } d_G(v)=2\}$ is even-sized.
\end{itemize}
\end{proposition}

\begin{proof}
We may assume that $G$ is connected. Notice that a $2$-edge-coloring of $G$ is odd if and only if every edge set $E_G(v)$ is monochromatic or dichromatic depending on whether $v$ is an odd vertex or a $2$-vertex of $G$.

Now $(i)\Rightarrow (ii)$ follows easily as moving around any given cycle $C\subseteq G$, there must occur an even number of color changes; in other words, $C$ must contain an even number (possibly $0$) of $2$-vertices of $G$.

To show $(ii)\Rightarrow (i)$, select a spanning tree $T$ rooted at an odd vertex $v_0$. First we color $E(T)$ as follows. Assign $E_T(v_0)$ with the color $1$, and repeatedly apply the following procedure until $E(T)$ becomes fully colored: choose a vertex $v\neq v_0$ that has just one incident edge already colored, say by a color $c\in\{1,2\}$; color the rest of $E_T(v)$ by the color $c$ (resp. $3-c$) if $d_G(v)$ is odd (resp. equal to $2$). This gives a $2$-edge-coloring $\varphi$ of $T$ that is dichromatic precisely at the $2$-vertices of $G$ which are not pendant in regard to $T$.

Let us extend $\varphi$ to $E(G)$. Consider an edge $e\in E(G)\backslash E(T)$, say $x$ and $y$ are its endvertices. Denote by $e_x$ and $e_y$ the (not necessarily distinct) edges of the $x$-$y$ path $P$ in $T$ that are incident with $x$ and $y$, respectively. Note that the equality $\varphi(e_x)=\varphi(e_y)$ holds if and only if an even number (possibly $0$) of internal vertices of $P$ are $2$-vertices in $G$. Therefore, since $P+e$ is a cycle, $\varphi(e_x)=\varphi(e_y)$ if and only if an even number (both or neither) of the vertices $x,y$ are $2$-vertices in $G$. So, we assign one of the colors $1,2$ to $e$ as follows: (1) if both $x,y$ are $2$-vertices in $G$, then set $\varphi(e)\neq \varphi(e_x)$; (2) if neither $x,y$ are $2$-vertices in $G$, then set $\varphi(e)=\varphi(e_x)$; if just one of the vertices $x,y$ is a $2$-vertex in $G$, say such is $x$, then set $\varphi(e)=\varphi(e_y)$. The resulting $\varphi$ is an odd $2$-edge-coloring of $G$ since on every edge set $E_G(v)$ it is monochromatic or dichromatic depending on whether $v$ is an odd vertex or a $2$-vertex.
\end{proof}

The above proof shows that the given characterization of odd $2$-edge-colorability within $\mathcal{S}$ is good and, in the affirmative, such a coloring can be found in polynomial time.

\begin{corollary}
    \label{odd 2-edge-colorability}
Let $G\in\mathcal{S}$. Then $\chi'_o(G)=2$ if and only if $G\notin\mathcal{O}$ and for every cycle $C$ of $G$ the set $\{v:v\in V(C) \text{ and } d_G(v)=2\}$ is even-sized.
\end{corollary}

In the remainder of the paper we provide a structural characterization of the class $\mathcal{S}_4$. The next result shall allow us to confine to $2$-connected graphs.

 \begin{proposition}
    \label{block tree}
 If $G\in\mathcal{S}$ is a connected graph, then  $\chi'_o(G)=4$ if and only if every block of $G$ belongs to $\mathcal{S}_4$ and for every cutvertex $v$ there is a unique block $B$ such that $d_B(v)$ is odd.

\end{proposition}
\begin{proof}
The essential part of our proof is to establish property $(P)$ below, which sheds some light on the structure of graphs $G\in\mathcal{S}$ of connectivity $\kappa(G)=1$ that require four colors for an odd edge-coloring.

\medskip

\noindent $(P)$ \textit{Let $v$ be a cutvertex of a connected graph $G\in\mathcal{S}$. If $G_1,\ldots,G_k$ are the $v$-lobes of $G$, then the following statements are equivalent:}
\begin{itemize}
\item[$(i)$]  \textit{$\{G_1,G_2,\ldots,G_k\}\subseteq\mathcal{S}_4$  and there is a unique $j$ such that $d_{G_j}(v)$ is odd; in particular, $d_{G_i}(v)=2$ for every $i\neq j$.}
\item[$(ii)$] \textit{$G\in\mathcal{S}_4$.}
\end{itemize}

Notice that, once the equivalence stated in $(P)$ is verified, the proposition may be derived by inducting on the number $t$ of cutvertices in $G$. Namely, the case $t=0$ is trivial, and the case $t=1$ follows immediately from $(P)$. For $t>1$, consider a cutvertex $v$ which is an internal leaf of the block-tree $B(G)$; in other words, $v$ is such that all but one of the blocks containing it are end-blocks of $G$. Let $G_1,\ldots,G_k$ be an enumeration of the $v$-lobes of $G$, so that $G_2,\ldots,G_k$ are end-blocks of $G$. Notice that the blocks of $G_1$ are precisely the blocks of $G$ that are $\neq G_2,\ldots,G_k$, whereas the cutvertices of $G_1$ are the cutvertices of $G$ distinct from $v$. In particular, the number of cutvertices of $G_1$ is $t-1$. Within the graph $G_1$, the vertex $v$ is an internal vertex of some block $B_1$, and thus $d_{G_1}(v)=d_{B_1}(v)$. By applying $(P)$ to the pair $G,v$ we deduce that $G\in\mathcal{S}_4$ if and only if  $G_1,G_2,\ldots,G_k\in\mathcal{S}_4$ and there is a unique $j$ such that $d_{G_j}(v)$ is odd, where it may happen that $j\neq1$. Combine this equivalence with the inductive hypothesis applied to $G_1$, and we are done.

\medskip

In what follows, we verify the property $(P)$ by proving the implications $(i)\Rightarrow(ii)$ and $(ii)\Rightarrow(i)$.
First we show the easy part, which is the direction $(i)\Rightarrow(ii)$. Arguing by contradiction, suppose that both $(i)$ and $\neg(ii)$  hold. Consider an odd $3$-edge-coloring $\varphi$ of $G$. For every $i=1,\ldots,k$, let $\varphi_i$ be the restriction of $\varphi$ to $E(G_i)$, that is, $\varphi_i=\varphi_{|E(G_i)}$. Clearly, each $\varphi_i$ is odd away from $v$. Moreover, since $\chi'_o(G_i)=4$, the edge-coloring $\varphi_i$  is not odd at $v$. Consequently, whenever $d_{G_i}(v)=2$ the edge set $E_{G_i}(v)$ is monochromatic under $\varphi_i$. However, then $\varphi_j$ must be an odd $3$-edge-coloring of $G_j$, a contradiction.

\medskip

Now let us prove direction $(ii)\Rightarrow(i)$. Assuming $(ii)$, note that degree $d_G(v)$ is odd by Proposition~\ref{4-edge-colorability}. We break the argument into several claims  which eventually lead to $(i)$.

\bigskip

\noindent \textbf{Claim 1.} \textit{If $G'\cup G''=G$ and $G'\cap G''=\{v\}$, then both orders $n(G'),n(G'')$ are odd. Moreover, if $d_{G'}(v)$ is even then any $3$-edge-coloring of $G'$ which is odd away from $v$ must be even at $v$; in particular, $G'$ is not odd $3$-edge-colorable.}

\medskip

\noindent By Proposition~\ref{even order}, $n(G)$ is odd. Hence, $n(G'),n(G'')$ are of the same parity. Suppose $n(G'),n(G'')$ are even. Then there exist spanning odd co-forests $H'$ and $H''$ of $G'$ and $G''$, respectively. Denote $F=G-E(H'\cup H'')$. Note that $F$ is a forest and $d_F(v)$ is odd. This enables construction of an odd $3$-edge-coloring of $G$ as follows: color $E(H')$ by $1$ and $E(H'')$ by $2$; color $E_F(v)$ by $3$; extend the coloring of $E_{F\cap G'}(v)$ to an edge-coloring of the forest $F\cap G'$ with color set $\{2,3\}$ that is odd away from $v$; similarly, extend the coloring of $E_{F\cap G''}(v)$ to an edge-coloring of the forest $F\cap G''$ with color set $\{1,3\}$ that is odd away from $v$. The obtained contradiction shows that $n(G'),n(G'')$ are both odd.

Suppose that $d_{G'}(v)$ is even and $G'$ admits a $3$-edge-coloring $\varphi'$ which is odd away from $v$ so that at least one, and hence two, of the colors are odd at $v$.
Assume the color set of $\varphi'$ is $\{1,2,3\}$ and the colors $1$ and $2$ are odd at $v$. Since $d_{G'}(v)$ is even, it follows that the color $3$ is even at $v$. We construct an accompanying edge-coloring $\varphi''$ of $G''$. In order to do so, consider an auxiliary graph $G^*=G''+vv^*$, where $v^*$ is a new vertex. Since $G^*$ is a connected graph of even order and the degree $d_{G^*}(v)$ is even, Proposition~\ref{even order} yields an odd edge-coloring of $G^*$ with color set $\{1,2,3\}$ such that $E_{G^*}(v)$ is colored by $2$ and $3$ with the edge $vv^*$ colored by $2$. Let $\varphi''$ be the restriction to $E(G'')$ of this coloring of $E(G^*)$. However, then $\varphi'\cup\varphi''$ is an odd $3$-edge-coloring of $G$. The obtained contradiction proves our point.\hfill$\diamond$

\bigskip

From the first part of Claim~1 it follows that every $v$-lobe of $G$ has an odd order. Next we use the last part of Claim~1 to show that the degree of $v$ is odd in regard to precisely one $v$-lobe.

\bigskip

\noindent \textbf{Claim 2.} \textit{There is a unique $j\in\{1,2,\ldots,k\}$ such that $d_{G_j}(v)$ is odd.}

\medskip

\noindent Since $d_G(v)$ is odd, so is $d_{G_j}(v)$ for some $j$. Suppose there are at least two such indices, say $j=1$ and $j=2$. It follows that $k\geq3$. Let $G'=G_{1}\cup G_{2}$ and $G''=\bigcup\{G_i:i=3,\ldots,k\}$. As both $G_1-v,G_2-v$ are connected graphs of even order, there exist spanning odd co-forests $H_1$ and $H_2$ of $G_1-v$ and $G_2-v$, respectively. By Proposition~\ref{consequences:forest}, take an edge-coloring of $\widehat{H_j}=G_j-E(H_j)$ with color set $\{1,2\}$ which is odd away from $v$ and so that the color $j$ is odd at $v$ in $\widehat{H_j}$, $j=1,2$.
Extend to $E(G')$ by coloring $E(H_1\cup H_2)$ with $3$. This furnishes an odd $3$-edge-coloring of $G'$. However, as $d_{G'}(v)$ is even, the obtained coloring contradicts the last part of Claim~1. \hfill$\diamond$

\bigskip

Proceed by showing that each $v$-lobe is a subdivision of an odd graph.

\bigskip

\noindent \textbf{Claim 3.} \textit{$G_i\in \mathcal{S}$ for every $i=1,2,\ldots,k$.}

\medskip

\noindent Supposing the opposite, there is a $v$-lobe $G_r$ such that $d_{G_r}(v)$ is even and $d_{G_r}(v)\geq4$. First we show that there exists of an even cycle $C\subseteq G_r$ with $v\in V(C)$. For this we may assume that every bouquet of $G_r$ incident with $v$ is a singleton (otherwise, there is a $2$-cycle through $v$). Since $d_{G_r}(v)\geq4$, consider a triplet $x,y,z\in N_{G_r}(v)$. If there is an even $x$-$y$ path in $G_r-v$, we are obviously done. So, as $G_r-v$ is connected, let $Q$ be an odd $x$-$y$ path. We exhibit an even path in $G_r-v$ going from $z$ to the set $\{x,y\}$.  In view of the connectedness of $G_r-v$, let $P$ be an odd $z$-$x$ path. On a traversal of $P$ from $z$ to $x$, say $w$ is the first vertex that belongs to $Q$.  Then $zPw\cup wQx$ and $zPw\cup wQy$ are paths (by the choice of $w$). Because $wQx$ and $wQy$ are of opposite parities, we have found an even $z$-$\{x,y\}$ path in $G_r-v$, which in turn yields an even cycle $C\subseteq G_r$ passing through $v$.

Next, we use the presence of the even cycle $C$ to show that there exists a $3$-edge-coloring $\varphi_r$ of $G_r$ which is odd away from $v$, and at $v$ two of the colors are odd. Consider the following local modification of $G_r$ that consists of splitting out $v$ entirely into $2$-vertices and then suppressing all of them except one: create a $2$-vertex $v'$ incident to the two edges forming $E_C(v)$; then arbitrarily split out the rest of $v$ into $2$-vertices and suppress them all (except $v'$). Every even vertex of the resulting connected graph $G_r'$ is a $2$-vertex, and $C$ is an even cycle (properly contained) in $G_r'$ that passes through its $2$-vertex $v'$. By Proposition~\ref{sufficient condition 3-edge-colorability}, $G_r'$ admits an odd $3$-edge-coloring. Returning to $G_r$, we obtain the desired $\varphi_r$. However, since $d_{G_r}(v)$ is even, the edge-coloring $\varphi_r$ contradicts with the second part of Claim~1. \hfill$\diamond$

\bigskip

The final piece of our argument is showing that no $v$-lobe is odd $3$-edge-colorable.

\bigskip

\noindent \textbf{Claim 4.} \textit{$\chi'_o(G_i)=4$ for every $i=1,2,\ldots,k$.}

\medskip

\noindent By Claim~1, no $v$-lobe $G_i$ with $d_{G_i}(v)=2$ is odd $3$-edge-colorable.
Suppose that the (unique) $v$-lobe $G_j$ having odd $d_{G_j}(v)$ is odd $3$-edge-colorable.
Take an odd edge-coloring $\varphi_j$ of $G_j$ with color set $\{1,2,3\}$ such that the color $1$ appears  on $E_{G_j}(v)$. Consider now an arbitrary $G_i$ with $i\neq j$. Letting $v^*$ be a new vertex, the graph $G_i+vv^*$ is connected and of even order, hence it admits an odd $3$-edge-coloring with color set $\{1,2,3\}$ under which the pendant edge $vv^*$ receives the color $1$. Denote by $\varphi_i$ the restriction to $E(G_i)$ of the constructed edge-coloring of $G_i+vv^*$. Note that $\varphi_i$ is odd away from $v$ and colors $E_{G_i}(v)$ with $1$, for otherwise $G_i$ would be odd $3$-edge-colorable. However then the union $\varphi_1\cup\cdots\cup\varphi_k$ is an odd $3$-edge-coloring of $G$, a contradiction.
\hfill$\diamond$

\bigskip

This completes the verification of $(ii)\Rightarrow(i)$, and thus settles the property $(P)$ and the proposition.
\end{proof}

A straightforward implication of Propositions~\ref{sufficient condition 3-edge-colorability} and~\ref{block tree} is that each connected member of $\mathcal{S}_4$ must have precisely one vertex of degree $2$.

\begin{corollary}
    \label{cor:single 2-vertex}
If $G\in\mathcal{S}_4$ is connected, then the set of its $2$-vertices is a singleton.
\end{corollary}
\begin{proof}
We prove the corollary by induction on $n(G)$. Assume first that $G$ has a cutvertex $v$. Then by the induction hypothesis, every $G_i$ with $d_{G_i}(v)=2$ in $(P)$ has the unique $2$-vertex $v$. So the unique $2$-vertex of $G_j$ with $d_{G_j}(v)$ odd is the unique $2$-vertex of $G$. Hence we may assume that
 $G$ is a block graph. Since $G\in \mathcal{S}_4$, its order is odd (by Proposition~\ref{even order}), and there is a $2$-vertex $v\in V(G)$.  Hence, $G$ is $2$-connected, that is, every pair of its vertices lie on a cycle. Therefore, by Proposition~\ref{sufficient condition 3-edge-colorability}, $v$ is the only $2$-vertex in $G$.
\end{proof}

In view of Proposition~\ref{2-connected} and Corollary~\ref{cor:single 2-vertex}, we are left with the task of determining which $2$-connected loopless graphs with a single $2$-vertex belong to $\mathcal{S}_4$. We proceed to describe a construction which shall enable us to narrow down the search to essentially $3$-edge-connected graphs.

\begin{figure}[ht!]
	$$
		\includegraphics[scale=0.8]{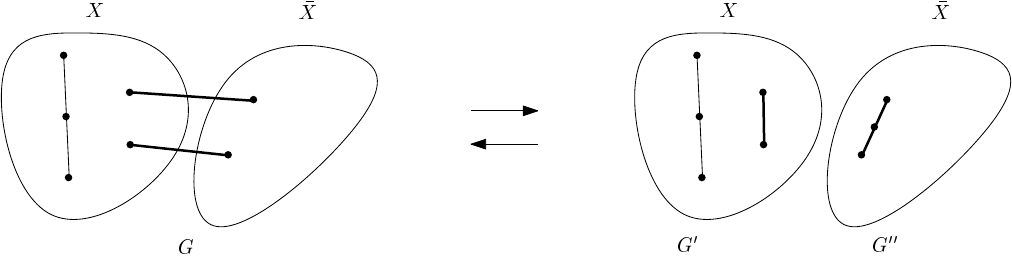}
	$$
	\caption{A $2$-connected graph $G$ (left), and graphs $G',G''$ (right). The fat edges form the symmetric difference $E(G)\oplus(E(G')\cup E(G''))$.}
	\label{fig:konstrukcija''}
\end{figure}

Consider a $2$-connected loopless graph $G$ that is obtainable from an odd graph by a single edge subdivision. Thus $\kappa(G)=\kappa'(G)=\delta(G)=2$. Assume $G$ is not essentially $3$-edge-connected, that is, let there be a nontrivial $2$-edge cut $\partial(X)$. Say the unique $2$-vertex of $G$ falls in $X$, and consider the graphs $G'$ and $G''$ constructed as follows (see also Figure~\ref{fig:konstrukcija''}):

\begin{itemize}
\item[$(i)$] $G'$ is derived from $G[X]$ by adding an edge between the two endvertices of $\partial(X)$ in $X$;
\item[$(ii)$] $G''$ is obtained from $G[\overline{X}]$ by introducing a new vertex and joining it with the two endvertices of $\partial(X)$ in $\overline{X}=V(G)\backslash X$; (equivalently, $G''$ is derived from $G$ by shrinking $X$ to a single new vertex, i.e. $G''=G/X$).
\end{itemize}

Notice that both $G',G''$ are $2$-connected and obtainable from odd graphs by single edge subdivisions.
Reversing the process, let $G',G''$ be disjoint $2$-connected loopless graphs that are obtained from odd graphs by single edge subdivisions. Break a selected edge of $G'$ into two half-edges, then remove the $2$-vertex of $G''$ along with its incident half-edges, and finally pair up and glue the half-edges emanating from $G'$ with the corresponding half-edges emanating from $G''$ so that two new (whole) edges are created. Call this process \textit{gluing} of $G'$ and $G''$ (with respect to a selected edge of $G'$). The  resulting graph $G$ is also $2$-connected and obtainable from an odd graph by a single edge subdivision, moreover it has a nontrivial $2$-edge cut.

\begin{figure}[ht!]
	$$
		\includegraphics[scale=0.8]{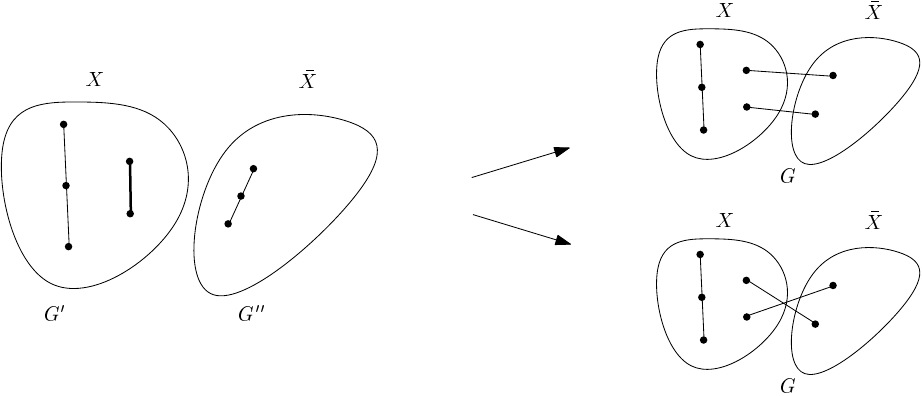}
	$$
	\caption{Graphs $G',G''$ (left), and the two ways to obtain a graph $G$ (right) by gluing $G',G''$ in respect of a selected edge of $G'$ (depicted as fat).}
	\label{fig:konstrukcija'''}
\end{figure}
We point out here that the result of gluing such a disjoint pair $G',G''$ is not unique, because of an apparent  $2$-fold freedom involved in the process: first, there is a freedom of choice due to the arbitrariness of the selected edge from $G'$; and second, there is freedom concerning the pairing the half-edges emanating from $G'$ with the half-edges emanating from $G''$ (cf. Figure~\ref{fig:konstrukcija'''}).

The importance of transforming $G$ into the pair $G',G''$ and vice versa comes from the following.

\begin{proposition}
    \label{glue}
Let $G$ be a $2$-connected loopless graph that is obtained from an odd graph by a single edge subdivision, and let $\partial(X)$ be a nontrivial $2$-edge cut in $G$ such that the unique $2$-vertex is in $X$. With $G',G''$ as described above, the following equivalence holds:
\begin{equation*}
\chi'_o(G)=4 \quad\text{ if and only if }\quad \chi'_o(G')=\chi'_o(G'')=4\,.
\end{equation*}
\end{proposition}
\begin{proof}
Assuming $\chi'_o(G)=4$, we argue by contradiction that $\chi'_o(G')=\chi'_o(G'')=4$. In view of Proposition~\ref{4-edge-colorability}, more than four colors are never required. Let $e\in E(G')\backslash E(G)$. Suppose $\chi'_o(G')\leq3$ and consider an odd $3$-edge-coloring $\varphi$ of $G'$. Extend the restriction $\varphi_{|E(G[X])}$ to $E(G)$ by using the color $\varphi(e)$ for $E(G)\backslash E(G[X])$. This gives an odd $3$-edge-coloring of $G$, a contradiction.

Suppose now that $\chi'_o(G'')\leq3$. We already know from Corollary~\ref{cor:single 2-vertex} (or from Proposition~\ref{even order}) that the graph $G'*e$, obtained from $G'$ by introducing a $2$-vertex on the edge $e$, is odd $3$-edge-colorable. However then an odd $3$-edge-coloring of $G$ arises by combining an odd $3$-edge-coloring of $G'*e$ and an odd $3$-edge-coloring of $G''$ with the same color set (after possibly permuting colors in the latter). This contradiction settles the issue that $\chi'_o(G)=4$ implies $\chi'_o(G')=\chi'_o(G'')=4$.

Let us show the reversed implication by contrapositive. Assume an odd $3$-edge-coloring of $G$ exists and consider its restrictions over the $2$-edge cut $\partial(X)$. If the two edges forming this cut are colored the same, then we have an odd $3$-edge-coloring of $G'$. Otherwise, if the two edges are colored differently, then an odd $3$-edge-coloring of $G''$ readily appears.
\end{proof}

We end this section with a property shared by all $2$-connected members of $\mathcal{S}_4$.

\begin{proposition}
    \label{lema:1}
If $G\in\mathcal{S}_4$ is a block graph, then $G$ can be obtained from a bipartite block odd graph by a single edge subdivision.
\end{proposition}

\begin{proof}
Since $G\in \mathcal{S}_4$, from Corollary~\ref{cor:single 2-vertex} it follows that there is a single $2$-vertex $v\in V(G)$. Our task is to prove that the graph $G\%v$, obtained from $G$ by suppressing $v$, is a bipartite block odd graph. The graph $G$ is $2$-connected. Hence, as $v$ is the only $2$-vertex of $G$, we have that $G\%v$ is a block odd graph. Concerning the bipartiteness of $G\%v$, by Proposition~\ref{sufficient condition 3-edge-colorability}, every cycle of $G$ passing through $v$ is odd. We are left to show that every cycle of $G$ that avoids $v$ is even.

Letting $N_G(v)=\{u,w\}$, suppose there is an odd cycle $C_o$ in $G-v$. By the $2$-connectedness of $G$, there exist two disjoint $\{u,w\}$-$V(C_o)$ paths, say a $u$-$u'$ path $P$ and a $w$-$w'$ path $Q$. Let $R$ denote the $u'$-$w'$ path along $C_o$ which is even (resp. odd) if $P$ and $Q$ have same (resp. opposite) parities. Then $P\cup R\cup Q$ is an even $u$-$w$ path in $G-v$. However, we have already established that no cycle through $v$ is even. This contradiction proves our point, that is, the graph $G\%v$ is a bipartite block odd graph.
\end{proof}

 \begin{figure}[ht!]
	$$
		\includegraphics[scale=0.75]{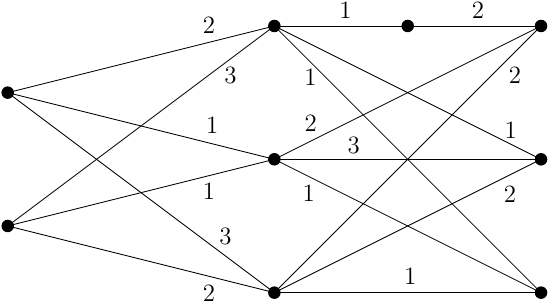}
	$$
	\caption{An odd $3$-edge-coloring of the graph obtained from $K_{3,5}$ by a single edge subdivision.}
	\label{fig:k35}
\end{figure}

Not every graph which can be obtained from a bipartite block odd graph by a single edge subdivision belongs in $\mathcal{S}_4$. For example, it can be readily seen from Figure~\ref{fig:k35} that the graph obtained from $K_{3,5}$ by a single edge subdivision is odd $3$-edge-colorable. On the other hand, in view of Theorem~\ref{thm:characterizationsubcubic}, this is not the case for the analogous graph obtained from $K_{3,3}$.

Note in passing that if $G',G''$ are disjoint $2$-connected graphs each obtainable from a bipartite odd graph by a single edge subdivision, then the result of any gluing of $G',G''$ is another such graph.  In the section we resolve the question which $2$-connected graphs obtainable from a bipartite odd graph by a single edge subdivision belong to the class $\mathcal{S}_4$.

\section{Characterization of $\mathcal{S}_4$}

As a result of Proposition~\ref{block tree}, we may confine to $2$-connected graphs. Let us denote by $\mathcal{F}$ the family defined inductively as follows:

\begin{itemize}
\item[$(a)$] every Shannon triangle of type $(2,1,1)$ and minimum degree $2$ belongs to $\mathcal{F}$;
\item[$(b)$] every graph that can be obtained by a single edge subdivision from a $3$-edge-connected bipartite cubic graph of order at least $4$  belongs to $\mathcal{F}$;
\item[$(c)$] every other graph $G$ in $\mathcal{F}$ can be constructed by taking disjoint members $G',G''\in\mathcal{F}$, and gluing them together.
\end{itemize}

\begin{figure}[ht!]
	$$
		\includegraphics[scale=0.9]{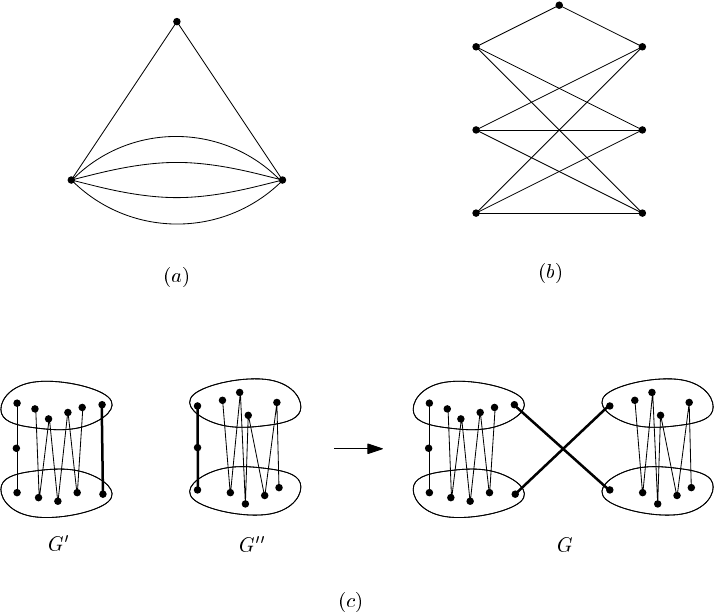}
	$$
	\caption{(a) A Shannon triangle of type $(2,1,1)$ with $\delta=2$; (b) A single edge subdivision of $K_{3,3}$; (c) Disjoint graphs $G',G''\in\mathcal{F}$ (left) and a graph $G$  (right) obtained by gluing $G',G''$ (the fat edges form the symmetric difference $(E(G')\cup E(G'')\oplus E(G)$).}
	\label{fig:family}
\end{figure}

We point out that the condition $\delta=2$ is included in $(a)$ in order to stay within $\mathcal{S}$. Clearly, $\mathcal{F}\subseteq \mathcal{S}$ and every member of $\mathcal{F}$ is a $2$-connected graph.  Note in passing that parts $(a)$,$(b)$ and $(c)$ of the above constructive definition of $\mathcal{F}$ are pairwise disjoint: every graph from $(a)$ is of order $3$ whereas every graph from $(b)$ or $(c)$ is of odd order at least $5$; every graph from $(b)$ is essentially $3$-edge-connected whereas every graph from $(c)$ has a nontrivial $2$-edge cut (cf. Figure~\ref{fig:family}).
The family $\mathcal{F}$ happens to be vital for our desired characterization. Namely, it turns out that a block graph $G$ belongs to $\mathcal{S}_4$ if and only if it belongs to $\mathcal{F}$, which we prove next.

\begin{theorem}
    \label{2-connected}
Let $G\in \mathcal{S}$ be a block graph. Then the following statements are equivalent:
\begin{itemize}
\item[$(i)$] $G\in \mathcal{F}$;
\item[$(ii)$] $G\in \mathcal{S}_4$.

\end{itemize}
\end{theorem}

\begin{proof}
We shall establish both $(i)\Rightarrow (ii)$ and $(ii)\Rightarrow (i)$. Since $\mathcal{F}\subseteq \mathcal{S}$, the former implication consists of showing that every graph $G\in\mathcal{F}$ has odd chromatic index $\chi'_o(G)=4$. So, in view of the equality~\eqref{eqn:shan} and Proposition~\ref{glue}, we only need to use the following fact: \textit{Every graph that can be obtained from a bipartite cubic graph by a single edge subdivision is not odd $3$-edge-colorable.} The proof of this is a straightforward double-counting argument (see~\cite{AtaPetSkr16} for the details).

\smallskip

The key ingredient for proving the implication $(ii)\Rightarrow (i)$ is provided by the next auxiliary result.

\begin{lemma}
    \label{lema:2}
Let $G\in\mathcal{S}$ be a $2$-connected and essentially $3$-edge-connected graph such that both the order $n(G)$ and the maximum degree $\Delta(G)$ are greater than $3$. Then $\chi'_o(G)\leq3$.
\end{lemma}

\medskip

\noindent \textit{Proof.} Arguing by contradiction, let $G$ be a minimal counter-example. By Proposition~\ref{lema:1} and assuming $v$ is the unique $2$-vertex of $G$, the graph $G\%v$ is bipartite; that is, every cycle through $v$ is odd and every cycle avoiding $v$ is even. Since, apart from $v$, every other vertex in $G$ is of odd degree we have the following.

\bigskip

\noindent \textbf{Claim 1.} \textit{No connected even subgraph $H\subseteq G$ satisfies that $v\in V(H)$ and $n(H)$ is even.}

\medskip

\noindent Arguing by contradiction, note that in the edge-complement $\widehat{H}=G-E(H)$, the vertex $v$ is isolated whereas every other vertex has an odd degree. Take an odd factor $K$ of $H$, and color $E(K)$ by $1$, $E(H)\backslash E(K)$ by $2$ and $E(G)\backslash E(H)$ by $3$. This gives an odd $3$-edge-coloring of $G$, a contradiction.
\hfill$\diamond$

\bigskip

In particular, it follows from Claim~1 that there is no pair of cycles $C,C'$, one of which passes through $v$, such that $V(C)\cap V(C')$ is a singleton, say $\{z\}$; call this formation a \textit{forbidden cycle pair at $z$} (cf. Figure~\ref{fig:forbidden}).  Several structural constraints arise from the absence of forbidden cycle pairs.

\begin{figure}[ht!]
	$$
		\includegraphics{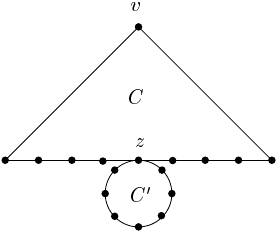}
	$$
	\caption{A forbidden cycle pair at a vertex $z$.}
	\label{fig:forbidden}
\end{figure}

\smallskip

Let $N_G(v)=\{u,w\}$. Since $G\in\mathcal{S}_4$, both $n(G)$ and  $\Delta(G)$ are odd and $\geq5$. Consider an arbitrary `large' vertex $z$, that is, a vertex of degree $d_G(z)\geq5$. By the $2$-connectedness of $G$, there exists a cycle $C\subseteq G$ such that $v,z\in V(C)$ and $|V(C)|\geq5$. Indeed, if $z\neq u,w$ then any cycle through $v$ and $z$ works; otherwise, select a vertex from $V(G)\backslash\{u,v,w\}$ and use a cycle passing through $v$ and that vertex. Let $P=C-v$ be the $u$-$w$ path that goes through $z$ and is contained within $C$ (it is not excluded that $z$ is an endvertex of $P$). We consider the collection $\mathcal{P}_z$ of paths $Q$ in $G-v$ such that $Q$ connects $z$ and another vertex of $P$ and $P\cap Q$ consists of these two vertices. Let us refer to the other endvertex of $Q\in\mathcal{P}_z$ as its \textit{ending}. We denote by $\mathrm{In}(Q)$ the set of \textit{internal vertices} of the path $Q$, those that are not its endvertices.\footnote{Not to be confused with `internal vertices of a block'.}

\bigskip

\noindent \textbf{Claim 2.} \textit{There is a mapping from $E_G(z)\backslash E_C(z)$ to $\mathcal{P}_z$, sending $e\mapsto Q_e$, such that $e\in E(Q_e)$. Moreover, for every such mapping it holds that}

\begin{equation*}
e\neq e' \quad \Rightarrow \quad Q_e \text{ and } Q_{e'} \text{ are internally disjoint, i.e., } \mathrm{In}(Q_e)\cap \mathrm{In}(Q_{e'})=\emptyset\,.
\end{equation*}

\medskip

\noindent Let $x$ be the other endvertex of $e$ (besides $z$). If $x\in V(P)$, all of the claimed is trivially true. Indeed, by then the path $Q_e$ is uniquely determined and $\mathrm{In}(Q_e)=\emptyset$ since $Q_e$ is the $1$-path with edge set $\{e\}$. Otherwise, if $x\notin V(P)$, then $x$ falls into a component, $K_e$, of $G-V(C)$. Note that then  $e\neq e'$ implies $K_e\neq K_{e'}$, for otherwise a forbidden cycle pair at $z$ (that includes $C$) is present (cf. Figure~\ref{fig:cycle pair}); in particular, $e$ and $e'$ are not parallel edges. From this readily it follows that $Q_e$ exists in this case as well. Namely, every edge in $\partial(V(K_e))$ has an endvertex in $K_e$ and an endvertex on $P$; moreover, $|\partial(V(K_e))|\geq3$ since $G$ is essentially $3$-edge-connected. Let us note in passing that neither $Q_e$ nor its ending are no longer uniquely determined (as we already established that no two edges from the edge cut $\partial(V(K_e))$ can have the same endvertex on $P$). Observe that $\mathrm{In}(Q_e)\subseteq V(K_e)$. Therefore, since $e\neq e'$ implies $K_e\neq K_{e'}$, we have that $e\neq e' \Rightarrow \mathrm{In}(Q_e)\cap \mathrm{In}(Q_{e'})=\emptyset$.\hfill$\diamond$

\begin{figure}[ht!]
	$$
		\includegraphics[scale=0.6]{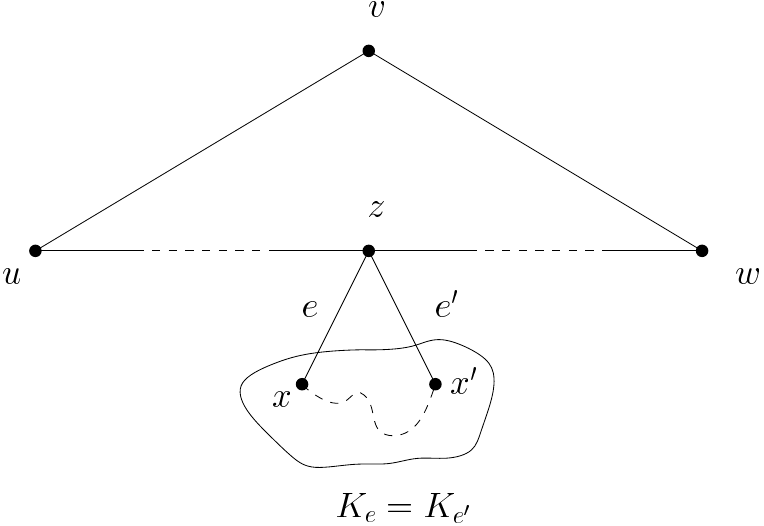}
	$$
	\caption{A forbidden cycle pair at $z$ if $K_e=K_{e'}$. Letting $x$ and $x'$ be the other endvertices (besides $z$) of $e$ and $e'$, respectively, any $x$-$x'$ path within the shared component combines with $e$ and $e'$ to produce a cycle.}
	\label{fig:cycle pair}
\end{figure}

\bigskip

Any subsequent use of notation $Q_e$ is to be understood in the context of Claim~2. We study next the following situation: $e,e'\in E_G(z)\backslash E_C(z)$ are distinct edges and $Q_e,Q_{e'}$ have endings on the same side of $P$ in respect of $z$.

\bigskip

\noindent \textbf{Claim 3.} \textit{Let $Q',Q''\in \mathcal{P}_z$ be internally disjoint, and have their respective endings lying on the same side of $z$ along $P$. Then $Q',Q''$ are $1$-paths and their shared ending is in $N_P(z)$.}

\medskip

\noindent First we show that $Q'$ and $Q''$ share the same ending. Arguing by contradiction, suppose their respective endings, say $z'$ and $z''$, differ. Without loss of generality, let $z'$ be an internal vertex of the subpath $zPz''$. Denote $C'=zPz'\cup Q'$ and $C''=zPz''\cup Q''$. Then $C\oplus C''$ and $C'$ constitute a forbidden cycle pair at $z$ (cf. Figure~\ref{fig:endings}). The obtained contradiction confirms that $Q',Q''$ have the same ending, say $z^*$. Note in passing that $z^*$ is another large vertex along $P$.

\begin{figure}[ht!]
	$$
		\includegraphics[scale=0.8]{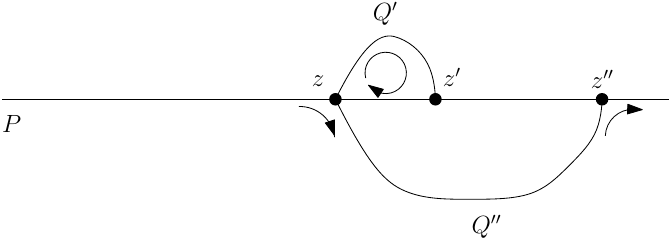}
	$$
	\caption{A detour from $P$ that yields a forbidden cycle pair at $z$.}
	\label{fig:endings}
\end{figure}

Next we prove that $Q',Q''$ are actually $1$-paths. For argument's sake, suppose $\mathrm{In}(Q')\neq\emptyset$. Then $\mathrm{In}(Q')$ is contained within a single component $K$ of $G-V(C)$. Note in passing that $V(K)\cap\mathrm{In}(Q'')=\emptyset$, by the proof of Claim~2. Consider the edge cut $\partial(V(K))$. The essential $3$-edge-connectedness of $G$ and Claim~2 together guarantee that there is a $V(P)$-$V(K)$ edge $f\notin E(Q')$ such that the endvertex of $f$ on $P$ is neither $z$ nor $z^*$. However, such an edge $f$ would  contradict the already established feature of shared path endings. Indeed, let $x$ and $y$ be the respective endvertices of $f$ in $V(P)$ and $V(K)$, and let $Q$ be a $y$-$\mathrm{In}(Q')$ path within $K$, say $y'$ is the other endvertex of $Q'$. Then each of the paths $zQ'y'\cup Q+f$ and $z^*Q'y'\cup Q+f$ is internally disjoint with $Q''$; moreover, $zQ'y'\cup Q+f\in\mathcal{P}_z$ and $z^*Q'y'\cup Q+f\in \mathcal{P}_{z^*}$. Hence, depending on the position of $x$ along $P$, at least one of the pairs $zQ'y'\cup Q+f,Q''$ and $z^*Q'y'\cup Q+f,Q''$ yields the mentioned contradiction  (either in regard to the large vertex $z$ or to the large vertex $z^*$; see Figure~\ref{fig:neighbor}). Consequently, both $Q'$ and $Q''$ are $z$-$z^*$ edges.

\begin{figure}[ht!]
	$$
		\includegraphics[scale=0.8]{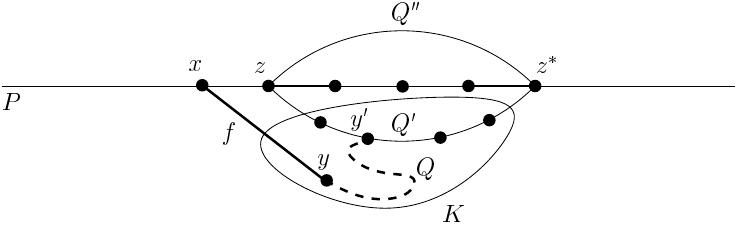}
	$$
	\caption{A pair of internally disjoint paths $z^*Q'y'\cup Q+f,Q''\in\mathcal{P}_{z^*}$ with distinct endings ($x$ and $z$, respectively) on the same side of $z^*$ along $P$.}
	\label{fig:neighbor}
\end{figure}

Finally, let us show that $z^*\in N_P(z)$, that is, $z^*$ is a neighbor of $z$ along $P$. Once again we argue by contradiction and evoke essential $3$-edge-connectedness. Suppose $\mathrm{In}(zPz^*)\neq\emptyset$, and let $H$ be the component of $G-(V(C)\backslash\mathrm{In}(zPz^*))$ that includes $\mathrm{In}(zPz^*)$. Since $|\partial(H)|\geq3$ there exists an edge $h\in \partial(H)\backslash E(P)$. It cannot be that $h$ has an endvertex within $\{z,z^*\}$. Indeed, for otherwise, $h$ would contradict the feature of shared path endings in regard to its endvertex in $\{z,z^*\}$ (see Figure~\ref{fig:neighbor}). Thus, without loss of generality, assume $h$ meets $P$ on the side of $z$ not including $z^*$. If $x$ is the endvertex of $h$ on $P$, then there is a path $R$ that starts at $x$ along $h$ and goes through $H$ until it reaches $P$ again, say at a vertex $y\in \mathrm{In}(zPz^*)$.

\begin{figure}[ht!]
	$$
		\includegraphics[scale=0.8]{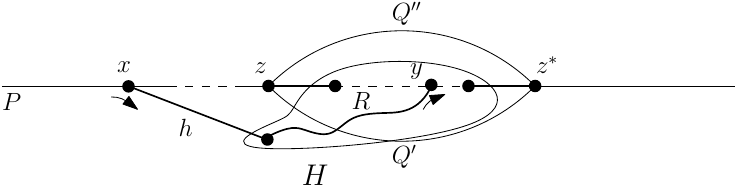}
	$$
	\caption{A detour from $P$ along $R$ yielding a forbidden cycle pair at $z^*$.}
	\label{fig:neighbor}
\end{figure}

\noindent Define $\overline{C}=xPy \cup R$ and $\overline{\overline{C}}=Q'\cup Q''$. Then $C\oplus \overline{C}$ and $\overline{\overline{C}}$ form a forbidden cycle pair at $z^*$. The obtained contradiction settles the claim.
\hfill$\diamond$

\bigskip

Since $z$ is a large vertex, we have that $|E_G(z)\backslash E_C(z)|\geq3$. Consequently, on at least one side along $P$ the vertex $z$ is incident with a $3^+$-bouquet $\mathcal{B}_{zz^*}$ (see Figure~\ref{fig:bouquet}); call it a \textit{large bouquet}. Thus, every large vertex lying on $P$ is incident with at least one large bouquet (shared with an adjacent large vertex along $P$). Moreover, every large bouquet incident with a vertex of $P$ is of this kind, for otherwise a forbidden cycle pair occurs. From Claims~2 and~3 it also follows that for each $e\in E_G(z)\backslash E_C(z)$, all the paths $Q_e$ have endings on the same side of $P$ with regard to $z$.

\begin{figure}[ht!]
	$$
		\includegraphics[scale=0.9]{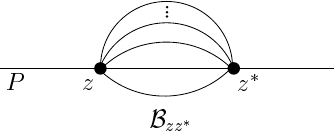}
	$$
	\caption{A large bouquet at $z$ (and at a neighbor $z^*$) along $P$.}
	\label{fig:bouquet}
\end{figure}

\bigskip

\noindent \textbf{Claim 4.} \textit{The multiplicity $\mu(G)=3$ whereas the maximum degree $\Delta(G)=5$. Moreover, every $5$-vertex on $P$ is incident with a $3$-bouquet, which it shares with a neighboring $5$-vertex along $P$.}

\medskip

\noindent We already noted prior to this claim that every large vertex along $P$ is incident with a large bouquet (shared with a large neighbor on $P$). Let us first  show that $\mu(G)=3$. Consider in $G$ a bouquet $\mathcal{B}$ of maximum size. Thus $|\mathcal{B}|\geq3$, implying that its endvertices are large vertices, say $z'$ and $z''$. So (by taking $z=z'$) we may assume that $\mathcal{B}$ is a large bouquet along $P$. Select two edges $e,f\in\mathcal{B}\backslash E(P)$. Now it is important to observe that in case $|\mathcal{B}|\geq4$ the graph $G-\{e,f\}$ satisfies all assumptions of the lemma. We proceed with clarifying this.

 If $|\mathcal{B}|\geq5$ then it is obvious that $G-\{e,f\}\in\mathcal{S}$ is a $2$-connected and essentially $3$-edge-connected graph having both order and maximum degree greater than $3$. On the other hand, supposing $|\mathcal{B}|=4$, the sets $E_G(z')\backslash(E(P)\cup \mathcal{B}_{z'z''})$ and $E_G(z'')\backslash(E(P)\cup \mathcal{B}_{z'z''})$ are even-sized. If neither of the vertices $z',z''$ is incident with another large bouquet along $P$, then (by Claims~2 and~3) the sets $E_G(z')\backslash(E(P)\cup \mathcal{B}_{z'z''})$ and $E_G(z'')\backslash(E(P)\cup \mathcal{B}_{z'z''})$ are actually empty. However, that would imply  $|\partial(\{z',z''\})|=2$, contradicting the essential $3$-edge-connectedness of $G$. So, at least one of $z',z''$ must be incident with a $3$-bouquet along $P$. This yields the same conclusion that $G-\{e,f\}$ satisfies all the assumptions of the lemma. Indeed, the $2$-connectedness, order and degree assumptions are clearly preserved. As for the essential $3$-edge-connectedness of $G-\{e,f\}$, suppose there is a nontrivial $2$-edge cut. Then $z',z''$ must be on different sides of this cut, and hence the cycle $C$ must have at least two edges in common with the cut. We have thus detected at least three edges in a $2$-edge cut, a contradiction.

  Therefore, $G-\{e,f\}$ is odd $3$-edge-colorable (by the minimality choice of $G$). But such a coloring of $E(G)\backslash\{e,f\}$ readily extends to an odd $3$-edge-coloring of $G$ by using for both $e,f$ one color already appearing on $\mathcal{B}\backslash\{e,f\}$, a contradiction. Hence, it must be that $|\mathcal{B}|=3$, confirming $\mu(G)=3$ and also showing that every large vertex along $P$ is incident with a $3$-bouquet.

Finally, suppose there is a large vertex $z$ of degree greater than $5$. It is incident with a $3$-bouquet $\mathcal{B}_{zz^*}$ along $P$, and has $|E_G(z)\backslash(E(P)\cup B_{zz^*})|\geq3$. In view of Claims~2 and~3, this inequality grants a $4^+$-bouquet along $P$ which is incident with $z$ and lies on the other side of $z^*$. However, such a bouquet contradicts with the already established equality $\mu(G)=3$. Consequently, $\Delta(G)=5$.
\hfill$\diamond$

\bigskip

So the $5$-vertices along $P$ come in pairs, each \textit{pair} consisting of two neighbors on $P$ which are the endvertices of a $3$-bouquet. Two such pairs $(z,z^*)$ and $(\bar{z},\bar{z}^*)$ are said to be \textit{successive} if the vertices $z,z^*,\bar{z},\bar{z}^*$ are in that relative order on a traversal of $P$ from $u$ to $w$ and $\mathcal{P}_{z^*}\cap\mathcal{P}_{\bar{z}}\neq\emptyset$ (cf. Figure~\ref{fig:successive}).

\begin{figure}[ht!]
	$$
		\includegraphics[scale=0.75]{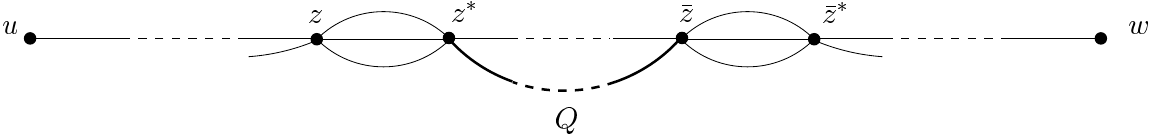}
	$$
	\caption{Successive pairs $z,z^*$ and $\bar{z},\bar{z}^*$. A $z^*$-$\bar{z}$ path $Q$ (depicted as fat) is internally disjoint from $P$, that is, it holds that $Q\in \mathcal{P}_{z^*}\cap\mathcal{P}_{\bar{z}}$.}
	\label{fig:successive}
\end{figure}

Consider a maximal sequence $\mathcal{Z}:(z_1,z_1^*),(z_2,z_2^*),\ldots,(z_n,z_n^*)$ of pairs (of $5$-vertices along $P$) subjected to the condition that the pairs $(z_i,z_i^*)$ and $(z_{i+1},z_{i+1}^*)$ are successive, for each $i=1,2,\ldots,n-1$. Select a path $Q_i\in \mathcal{P}_{z_i^*}\cap\mathcal{P}_{z_{i+1}}$, for each $i=1,2,\ldots,n-1$. Also take paths $Q_0\in\mathcal{P}_{z_1}\backslash\mathcal{P}_{z_1^*}$ and $Q_n\in\mathcal{P}_{z_n^*}\backslash\mathcal{P}_{z_n}$. If $x$ and $y$ are the other endvertices of $Q_0$ and $Q_n$, respectively, besides $z_1$ and $z_n^*$. By the maximality choice of $\mathcal{Z}$, the vertices $x$ and $y$ are $3$-vertices of $G$. Moreover, the vertices $x,z_1,z_1^*,z_2,z_2^*,\ldots,z_n,z_n^*,y$ are in that relative order on traversal of $P$ from $u$ to $w$. Consider the subgraph $H$ of $G$ defined as follows (cf. Figure~\ref{fig:subgraph}): $$H=C\cup\bigcup\{Q_i:i=0,\ldots,n\}\cup\bigcup\{B_{z_iz_i^*}:i=1,\ldots,n\}\,.$$

\begin{figure}[ht!]
	$$
		\includegraphics[scale=0.7]{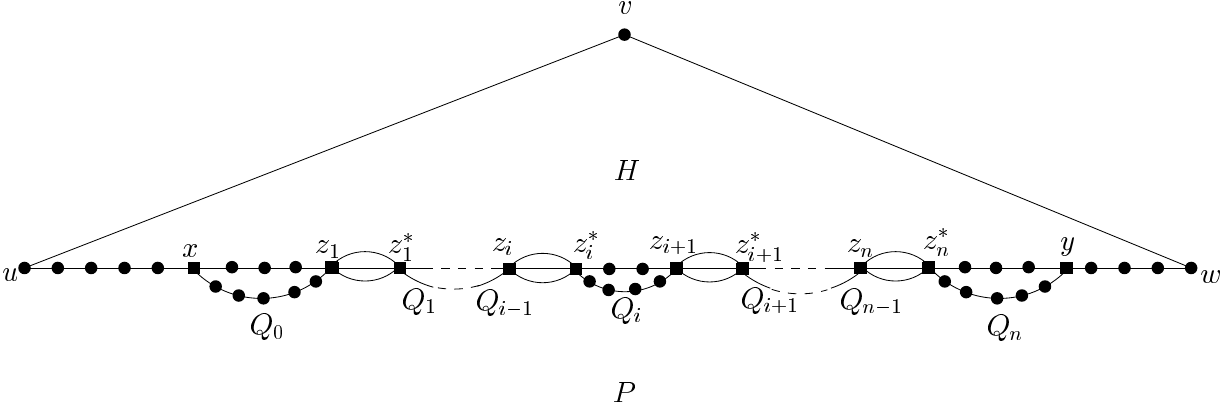}
	$$
	\caption{A sketch of the subgraph $H\subseteq G$.}
	\label{fig:subgraph}
\end{figure}

Denote $Z=\{x,z_1,z_1^*,\ldots,z_n,z_n^*,y\}$.  Observe that every vertex from $Z$ has the same degree in regards to both $H$ and $G$. Thus, every vertex from $Z$ is isolated in the edge-complement $\widehat{H}=G-E(H)$. Every other vertex of $H$ has degree $2$. Moreover, the set $E(H)$ has the following important features: $(i)$ it contains at least one $3$-bouquet (surely $\mathcal{B}_{z_1z_1^*}$ is such), and $(ii)$ if all $3$-bouquets are to be removed from $H$, then all that remains is the path $xPuvwPy$ and a collection of even pairwise disjoint cycles $C_0,C_1,\ldots,C_{n-1},C_n$, where the cycle $C_i$ consists of the path $Q_i$ and a suitable portion of $P$.

The above mentioned features of $E(H)$ enable the following construction of a particular edge-coloring $\varphi_H$ of $H$ with color set $\{1,2,3\}$: start by taking a proper edge-coloring of the path $xPuvwPy$ with color set $\{1,2\}$; for the remaining two uncolored edges at $x$ (belonging in $C_0$) use the already appearing color ($1$ or $2$), and then extend to an edge-coloring of $C_0$ with color set $\{1,2\}$ which is proper at each vertex $\neq x,z_1$; similarly,  for the remaining two uncolored edges at $y$ (belonging in $C_n$) use the already appearing color ($1$ or $2$), and extend to an edge-coloring of $C_n$ with color set $\{1,2\}$ which is proper at each vertex $\neq y,z_n^*$; proceed by using all three colors $1,2,3$ on each of the $3$-bouquets $\mathcal{B}_{z_1z_1^*},\ldots,\mathcal{B}_{z_nz_n^*}$; finally, to each of the remaining uncolored (even and disjoint) cycles $C_1,C_2,\ldots,C_{n-1}$ apply an  edge-coloring with color set $\{1,2\}$ which is proper at each vertex outside $Z$ (i.e., alternate here between the colors $1$ and $2$), whereas the coloring is improper at each vertex from $Z$ (i.e., repeat here the same color). Notice that the color $3$ occurs only on edges having both endvertices in $Z\backslash\{x,y\}$. Moreover, $\varphi_H$ is monochromatic (and thus odd) at each of the vertices $x,y$.

Now extend the above constructed $\varphi_H$ from $E(H)$ to $E(G)$ by coloring $E(G)\backslash E(H)$ with $3$. Since every vertex from $Z$ is isolated in $\widehat{H}$, the described extension is an odd $3$-edge-coloring of $G$. This contradiction settles Lemma~\ref{lema:2}. \qed

\medskip

The freshly proved Lemma~\ref{lema:2}, combined with Propositions~\ref{glue} and~\ref{lema:1}, yields the implication $(ii)\Rightarrow (i)$. Indeed, arguing by contradiction, let $G\in\mathcal{S}_4\backslash \mathcal{F}$ be a block graph of minimum order $n(G)$. Taking into account Proposition~\ref{glue}, it is implied by part $(c)$ of the construction of $\mathcal{F}$ that $G$ is essentially $3$-edge-connected, besides being $2$-connected. Consequently, by Lemma~\ref{lema:2}, it holds that $n(G)=3$ or $\Delta(G)=3$. However, if $n(G)=3$ then $G$ must be a Shannon triangle of type $(2,2,1)$ with $\delta(G)=2$; hence $G\in \mathcal{F}$ (due to part $(a)$ of the construction). Otherwise, if $\Delta(G)=3$ then Proposition~\ref{lema:1} assures that $G\in\mathcal{F}$ (due to part $(b)$ of the construction). The obtained contradiction settles the implication $(ii)\Rightarrow (i)$, which completes the proof of Theorem~\ref{2-connected}.
\end{proof}


\bigskip

  Finally, we arrive at the main result, which succinctly summarizes our findings.

\begin{theorem}
    \label{thm}
Let $G\in \mathcal{S}$ be a connected graph. If $(\mathcal{B},\mathcal{V})$ is the bipartition of the block-tree $B(G)$ of $G$, where $\mathcal{B}$ is the set of blocks and $\mathcal{V}$ the set of cutvertices of $G$, the following holds:
\begin{equation*}
\chi'_o(G)=
\begin{cases}
1 & \text{\quad if \,} G \text{ is odd} \,;\\
2 & \text{\quad if \,} G \text{ has 2-vertices, with an even number of them on each cycle}\,;\\
4 & \text{\quad if \,} \mathcal{B}\subseteq\mathcal{F} \text{ and for every }  v\in\mathcal{V} \text{ there is a unique } B\in\mathcal{B} \text{ with odd } d_B(v)\,;\\
3 & \text{\quad otherwise}\,.
\end{cases}
\end{equation*}
\end{theorem}
\begin{proof}
Straightforward from Corollary~\ref{odd 2-edge-colorability}, Proposition~\ref{block tree} and Theorem~\ref{2-connected}.
\end{proof}

\section{Further work}

It is implied by Theorem~\ref{thm} that the problem of determining the odd chromatic index of a subdivision of an odd graph is solvable in polynomial time. In view of Theorems~\ref{odd 4-edge-colorability} and~\ref{Kano}, the complexity questions concerning the value of the odd chromatic index of general graphs amount to deciding on their odd $3$-edge-colorability. As already mentioned in Section~1.2, this question is still open.
Let us propose the study of a related set of questions arising from the following line of reasoning. The class $\mathcal{S}$ can be captured by using the notion of \textit{maximum even degree}, defined as follows. Let $\Delta_{\mathrm{even}}(G)$ denote the maximum even value among the vertex degrees of $G$. Thus $\mathcal{S}=\{G: G \text{ is a loopless graph with }\Delta_{\mathrm{even}}(G)\leq2\}$, and Theorem~\ref{thm} tells that the problem of determining $\chi'_o(G)$ whenever $\Delta_{\mathrm{even}}(G)\leq2$ can be efficiently solved. For every $k=0,1,2,\ldots$, let  $\mathcal{S}^{(2k)}=\{G: G \text{ is a loopless graph with }\Delta_{\mathrm{even}}(G)\leq2k\}$. So, by ignoring isolated vertices, $\mathcal{S}^{(0)}=\mathcal{O}$; and obviously  $\mathcal{S}^{(2)}=\mathcal{S}$. We find the next question interesting.

\begin{question}
    \label{q1}
Is the decision problem whether a graph $G\in \mathcal{S}^{(4)}$ has $\chi'_o(G)\leq3$ solvable in polynomial time?
\end{question}

A positive answer to Question~\ref{q1} would open
the door for considering the following more general problem.

\begin{question}
    \label{q2}
Given a positive integer $k$, is the decision problem whether a graph $G\in \mathcal{S}^{(2k)}$ has $\chi'_o(G)\leq3$ solvable in polynomial time?
\end{question}

\begin{figure}[ht!]
	$$
		\includegraphics[scale=1.1]{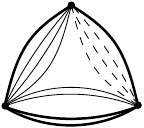}
	$$
	\caption{A Shannon triangle of type $(2,2,2)$ that requires three colors for an $\mathcal{S}$-edge-coloring. The edges falling in distinct color classes of an optimal coloring are respectively depicted as dashed, normal and heavier.}
	\label{fig:sh}
\end{figure}

Another possible field of study is to consider $\mathcal{S}^{(2k)}$-edge-colorability of graphs for a fixed positive integer $k$ (instead of odd edge-colorability); that is, define a new type of edge-coloring by requiring that each color class is a member of $\mathcal{S}^{(2k)}$ (rather than of $\mathcal{O}=\mathcal{S}^{(0)}$). Say the corresponding index (representing the minimum sufficient number of colors) is $\chi'_{\mathcal{S}^{(2k)}}(G)$. For example, it is readily observed that $\chi'_{\mathcal{S}}(W_4)=2$ as opposed to $\chi'_o(W_4)=4$. Similarly, if $G$ is a Shannon triangle of type $(2,2,2)$, then $\chi'_{\mathcal{S}}(G)\leq3$ in contrast to $\chi'_o(G)=6$ (cf. Figure~\ref{fig:sh}).
Note that there are graphs requiring at least four colors for an $\mathcal{S}$-edge-coloring. Namely, every Shannon triangle $G$ of type $(2,2,1)$ and multiplicity $\mu(G)\geq3$ has $\chi'_{\mathcal{S}}(G)=4$ (cf. Figure~\ref{fig:s}).

\begin{figure}[ht!]
	$$
		\includegraphics[scale=1.1]{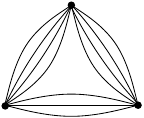}
	$$
	\caption{A Shannon triangle of type $(2,2,1)$ that requires four colors for an $\mathcal{S}$-edge-coloring. Its bouquets are of size $4, 4$ and $3$, respectively.}
	\label{fig:s}
\end{figure}

We are tempted to end our discussion here with the following.

\begin{conjecture}
    \label{conj:3}
If $G$ is a connected loopless graph that is not a Shannon triangle of type $(2,2,1)$ and multiplicity $\mu(G)\geq3$, then $\chi'_{\mathcal{S}}(G)\leq3$.
\end{conjecture}

One wonders whether the bound $3$ in Conjecture~\ref{conj:3} may drop to $2$ if $\mathcal{S}$ is replaced  with a certain $\mathcal{S}^{(2k)}$ of sufficiently large $k$. Understandably, the list of excluded graphs might become longer.


\bigskip
\noindent
{\bf Acknowledgements.}
This work is partially supported by ARRS Program P1-0383 and ARRS Projects J1-1692 and J1-3002.

\bibliographystyle{plain}

\end{document}